\documentclass[]{amsart}
\usepackage[margin=3cm]{geometry}
\usepackage{orcidlink}
\usepackage{supertabular}
\usepackage{amsfonts,latexsym,rawfonts,amsmath,amssymb, amsthm}
\usepackage{longtable,setspace}
\usepackage{algorithm2e}
\usepackage{caption}
\usepackage[T1]{fontenc}
\usepackage{latexsym,lscape,rawfonts}
\usepackage{graphicx}
\usepackage{array}
\newcolumntype{P}[1]{>{\centering\arraybackslash}p{#1}}
\newcolumntype{M}[1]{>{\centering\arraybackslash}m{#1}}

\usepackage{mathrsfs, enumitem}
\usepackage[title]{appendix}
\usepackage{hyperref}[backref]

\usepackage{galois}
\usepackage[all,cmtip]{xy}
\usepackage{extarrows,color}
\usepackage{times}

\newtheorem{thm}{Theorem}[section]

\newtheorem*{fixed point criterion}{Fixed point criterion}

\newtheorem{lem}[thm]{Lemma}

\newtheorem{claim}{Claim}

\theoremstyle{definition}
\newtheorem{defn}[thm]{Definition}

\newtheorem{ques}[thm]{Question}
\theoremstyle{remark}
\newtheorem{rem}[thm]{Remark}
\numberwithin{equation}{section}

\newcommand{\N}{\mathbb N}
\newcommand{\RR}{\mathcal R}

\newcommand{\gs}{\mathcal{S}}
\newcommand{\W}{\mathcal{W}}
\newcommand{\nf}{\mathfrak{nf}}
\newcommand{\cl}{\mathrm{CL}}
\newcommand{\rs}{\mathfrak{rs}}
\newcommand{\llfr}{\mathrm{LLFR}}

\title{An explicit bound for the conjugator length function of a surface group}

	\author[Ke Wang]{Ke Wang\orcidlink{0009-0005-7108-3725}}
	\address{School of Mathematics and Statistics, Xi'an Jiaotong University, Xi'an 710049, P. R. China}
	\email{keqiyehuopo@stu.xjtu.edu.cn}
	
	\author[Qiang Zhang]{Qiang Zhang\orcidlink{0000-0001-6332-5476}}
	\address{School of Mathematics and Statistics, Xi'an Jiaotong University, Xi'an 710049, P. R. China}
	\email{zhangq.math@mail.xjtu.edu.cn}

	\thanks{The authors are partially supported by National Natural Science Foundation of China (No. 12471066).}
	
	\date{\today}
\date{\today}

\begin{document}
\begin{abstract}
For a surface group $\pi_1(\Sigma_g)=\langle c_1,\dots , c_{2g}\mid c_1\cdots c_{2g}c_1^{-1}\cdots c_{2g}^{-1}\rangle$ with genus $g\geq 2$, we provide an explicit bound $n-1\leq \cl(2n)=\cl(2n+1)\leq n+8g-1$ for the conjugator length function $\cl:\N\to\N$ of $\pi_1(\Sigma_g)$ via a detailed analysis of conjugation reductions. 
\end{abstract}
\maketitle

\section{Introduction}
In the early 20th century, Dehn introduced three important problems: the word problem, the conjugacy problem \cite{De12} and the isomorphism problem. Among them, the conjugacy problem has become a fundamental decision problem in group theory today, asking whether two elements in a finitely generated group are conjugate. 

A key quantitative aspect of the conjugacy problem is the \emph{conjugator length function} $\cl:\N\to\N$. For two conjugate elements $u,v$, let $\cl(u,v)$ denote the minimal word length $|x|$ of conjugator $x$ s.t. $x^{-1}ux=v$. Then $\cl(n)$ is defined to be the least integer $N$ such that $\cl(u,v)\leq N$ for any conjugate elements $u$ and $v$ with word length $|u|+|v|\leq n$. There is a lot of research on the conjugator length function. The groups involved include hyperbolic groups \cite{AB23, BH05, Ly89}, relatively hyperbolic groups \cite{AS16}, Thompson's groups \cite{BM23}, solvable groups and lattices of semisimple Lie groups \cite{Sa12},  etc. Furthermore, Bridson and Riley wrote a series of articles to study the conjugator length function in finitely presented groups. In \cite{BR25a, BR25b}, they proved that for any $d\in\N$, there are finitely presented groups for which $\cl(n)\simeq n^d$. In \cite{BR25c}, they proved that the set of exponents $e$ for which there exists a finitely presented group with $\cl(n)\simeq n^e$ is dense in $[2,\infty)$, and in \cite{BR25d}, they constructed the first examples of finitely presented groups where the conjugator length function is exponential. 

In (Gromov) hyperbolic groups, the conjugator length function was proven to be linear by Lys\"enok in \cite{Ly89}: a hyperbolic group $\Gamma$ (with respect to any fixed finite generating set) has a conjugator length function $\cl(n)\leq C_0n+C_1$. In \cite{BH05}, Bridson and Howie proved that the coefficient $C_0$ can be taken as $1$ for hyperbolic groups. In \cite{AB23}, Abbott and Behrstock proved that in a hierarchically hyperbolic group, the conjugator length function for the conjugate Morse elements is also linear. 

In this paper, we focus on \emph{surface groups}, that is, the fundamental group $\pi_1(\Sigma_g)$ of a closed orientable surface $\Sigma_g$ of genus $g\geq 2$. Let $G$ be a surface group with the generating set $\gs=\{c_1,\dots,c_{2g}\}$ and the following \emph{symmetric presentation}
\begin{equation}\label{symmetric presentation}
G=\pi_1(\Sigma_g)=  \left\langle c_1,\dots , c_{2g}\mid c_1\cdots c_{2g}c_1^{-1}\cdots c_{2g}^{-1} \right\rangle, \qquad g\geq 2.
\end{equation} 
By presenting a detailed investigation of reductions of conjugation in the surface group $G$, we show that the linear coefficient $C_0$ can be taken as $1/2$ (which is sharp) in $G$. 
More precisely, our main result is as follows.
\begin{thm}\label{main thm1}
  Let $G$ be a surface group with a symmetric presentation (\ref{symmetric presentation}).
Then the conjugator length function $\cl:\N\to\N$ of $G$ satisfies
$$n-1\leq \cl(2n)=\cl(2n+1)\leq n+8g-1$$
for any genus $g\geq 2$.
\end{thm}

The paper is organized as follows. In Section \ref{sect 2}, we introduce some basic definitions and notations. In Section \ref{sect 3} and Section \ref{sect 4}, we establish several lemmas on conjugate reductions. Finally, in Section \ref{sect 5}, we give the proof of our main result.

\section{Preliminaries}\label{sect 2}

In this paper, we mainly study the conjugator length function of surface groups, based on the research of reductions in surface groups of a given order \cite{WZZ25}. We first introduce some definitions and notations used therein (see \cite[Section 2]{WZZ25} for more details). Throughout this paper, unless otherwise stated, let $G$ be a surface group with symmetric presentation (\ref{symmetric presentation}):
$$G=\pi_1(\Sigma_g)=  \left\langle c_1,\dots , c_{2g}\mid c_1\cdots c_{2g}c_1^{-1}\cdots c_{2g}^{-1} \right\rangle, \quad g\geq 2.$$
\subsection{Word length}
A \emph{word} in the alphabet $\gs^\pm:=\gs\cup\gs^{-1}=\{c_i^{\varepsilon}\mid 1\leq i\leq 2g, ~ \epsilon=\pm 1\}$ means a finite sequence $\overline{x_1 x_2\cdots x_n}$ with $x_i\in\gs^\pm,~i=1,\dots,n$. Here we use the overlined one to distinguish a word and a group element, i.e., $\overline{x_1\cdots x_n}$ is a word and $x_1\cdots x_n$ is a group element. Especially, we denote the \emph{empty word} as $1$. We denote the \emph{word set} in the alphabet $\gs^\pm$ as $\W(\gs)$. Furthermore, we write $\overline{A}\subset \overline{B}$ and $a\in\overline{A}$ if the word $\overline{A}$ is a subword of $\overline{B}$ and the letter $a$ is a letter appearing in $\overline{A}$, respectively.
\begin{rem}\label{rem word representation}
For a better explanation, in a word, we always use lowercase letters to represent letters in $\gs^\pm$ and use uppercase letters to represent subwords. For example, in the word $\overline{xyWa_1V_1z^{-1}}$, we have $x,y,a_1,z^{-1}\in\gs^\pm$ and $\overline{W},\overline{V}\in\W(\gs)$.  
\end{rem}
For a word $\overline{W}=\overline{x_1\cdots x_n}\in\W(\gs)$, we denote the length of $\overline{W}$ by $|\overline{W}|_\gs=n$. The (word) length of an element $u\in G$ with respect to the generating set $\gs$ is defined as 
$$|u|_\gs:=\min\{n\mid u=x_1\cdots x_n\in G,~x_i\in\gs^\pm\}.$$
When the generating set is clear, we always omit the subscript $\cdot_\gs$.

\subsection{Length-lexicographical order and normal form}
If we give the following order of the generators
\begin{equation}\label{order of generators}
    c_{2g}^{-1}\succ \cdots \succ c_2^{-1}\succ c_1^{-1}\succ c_1\succ c_2\cdots\succ  c_{2g},
\end{equation}
there will be a length-lexicographical order "$\succ$" on the \emph{word set} $\W(\gs)$ of $G$: for any $\overline{X},\overline{Y}\in \W(\gs)$,
\begin{equation*}
\begin{array}{ccccl}
 \overline{X} &\succ& \overline{Y} &\Longleftrightarrow&\left\{\begin{array}{ll}|\overline{X}|>|\overline{Y}|\qquad\mathrm{or}  \\ |\overline{X}|=|\overline{Y}| ~\mathrm{with}~ 
 \overline{X}=\overline{AxB}, ~\overline{Y}=\overline{AyC}
~\mathrm{for}~  x \succ y\in \gs^{\pm}\end{array}\right..\end{array}
\end{equation*}
Then, for an element $x\in G$, there is a unique \emph{normal form} $\nf(x)$ of $x$, which is the word with the minimal order that represents $x$. 

For a word $\overline{X}=\overline{x_1\cdots x_n}\in\W(\gs)$, we say it is \emph{irreducible} if it is a normal form; say it is \emph{freely reduced} if $x_i\neq x_{i+1}^{-1}$ for every $1\leq i\leq n-1$. Furthermore, we say it is \emph{cyclically irreducible} (resp. \emph{cyclically freely reduced}) if its every cyclic permutation (i.e. $\overline{x_i x_{i+1}\cdots x_n x_1x_2\cdots x_{i-1}},~1\leq i\leq n$) is irreducible  (resp. \emph{freely reduced}).

For any $x,y\in G$, we write $x\sim y$ if $x$ and $y$ are conjugate, and denote the conjugacy class of $x$ by $[x]$. Let $\nf([x])$ denote the \emph{normal form of the conjugacy class} $[x]$, defined as the unique cyclically irreducible word of the minimal order that represents an element in $[x]$.  

\subsection{$\llfr$, $S_{(i)}$ and reducing-subword pair}
\begin{defn}[$\llfr$]
Denote $\overline{R}:=\overline{c_1c_2\cdots c_{2g}c_1^{-1}c_2^{-1}\cdots c_{2g}^{-1}}$ and denote the word set of cyclic permutation of $\overline{R}$ and $\overline{R^{-1}}$ as
$$\RR:=\{\overline{b_{i}\cdots b_{4g}b_1\cdots b_{i-1}}\mid \overline{b_1\cdots b_{4g}}=\overline{R}\mbox{ or }\overline{R^{-1}},~1\leq i\leq 4g\}.$$
We say a word is a $k$-\emph{fractional relator} if it is equal to $\overline{b_1\cdots b_{k}}$ for some $\overline{b_1\cdots b_{4g}}\in\RR$ with $2\leq k\leq 4g$. Moreover, if a $k$-fractional relator $\overline{b_1\cdots b_{k}}$ is a subword of $$\overline{W}=\overline{w_1\cdots w_{m}b_1\cdots b_{k}w'_1\cdots w'_n}\in\W(\gs),$$  such that neither $\overline{ w_{m}b_1\cdots b_{k}}$ nor $\overline{b_1\cdots b_{k}w'_1}$ is a $(k+1)$-fractional relator, then we call $\overline{b_1\cdots b_k}$ a $k$-\emph{locally longest fractional relator} (\emph{LLFR}) of $\overline{W}$. 
\end{defn}

Note that for any $\overline{b_1\cdots b_{4g}}\in \RR$,  the following identities hold, where all subscripts are modulo $4g$:
	\begin{enumerate}
		\item $b_{k}=b_{k\pm2g}^{-1}$ for any $k=1,2,\dots,4g$;
		\item $b_{k}b_{k+1}\cdots b_{k+2g-1}=b_{k+2g-1}\cdots b_{k+1}b_{k}$ for any $k=1,2,\dots,4g$.
    \end{enumerate}

\begin{defn}[Words of type $S_{(i)}$]
   Let 
	\begin{enumerate}
		\item $S_{(1)}=\{\overline{b_1b_1^{-1}}-1\mid \overline{b_1\cdots b_{4g}}\in \RR\}$;
		\item $S_{(2,k)}=\{\overline{b_1\cdots b_{k}}-\overline{b^{-1}_{4g}\cdots b^{-1}_{k+1}}\mid \overline{b_1\cdots b_{4g}}\in \RR\}, ~2g+1\leq k\leq 4g$; 
		\item $S_{(3,t)}=\{\overline{b_1(b_2\cdots b_{2g})^tb_{2g+1}}-\overline{(b_{2g}\cdots b_{2})^t}\mid \overline{b_1\cdots b_{4g}}\in \RR\}, t\geq 2$;
		\item $S_{(4,t)}=\{\overline{b_1(b_2\cdots b_{2g})^t}-\overline{(b_{2g}\cdots b_{2})^tb_1}; ~\overline{(b_1\cdots b_{2g-1})^tb_{2g}}-\overline{b_{2g}(b_{2g-1}\cdots b_{1})^t}\mid \overline{b_1\cdots b_{4g}}\in \RR, b_1\succ b_{2g}\},t\geq 1$.
	\end{enumerate}
We will denote $S_{(i,k)}$ by $S_{(i)}$ ($i=2,3,4$) if there is no need to emphasize the second subscript $k$. Note that every polynomial in $S_{(i)}$ has a uniform form "$\overline{A}-\overline{B}$" for some $\overline{A},\overline{B}\in\W(\gs)$, and we call this $\overline{A}$ the \emph{leading word}.

We say a word $\overline{V}$ is \emph{of type} $S_{(i)}$ if it is the leading word of a polynomial in the set $S_{(i)}$. We can interpret this $S$-set (i.e. $\cup S_{(i)}$) as a rewriting system to define $S_{(i)}$ reducing $\overline{W_1}=\overline{CAD}$ to $\overline{W_2}=\overline{CBD}$ if $(\overline{A}-\overline{B})\in S_{(i)}$. We denote this process as
$$\overline{W_1}\xrightarrow{S_{(i)}}\overline{W_2}.$$
We say a word $\overline{W}$ is \emph{$S$-reduced} if $\overline{W}$ does not contain a subword equal to a leading word in $S$-set, and say a word is cyclically $S$-reduced if all its cyclic permutation are $S$-reduced.
\end{defn}

According to \cite[Lemma 2.8]{WZZ25}, we have the following lemma.
\begin{lem}
    $S$-reduced is equivalent to irreducible. Therefore, cyclically $S$-reduced is equivalent to cyclically irreducible.
\end{lem}

Furthermore, by \cite[Lemma 3.2]{WZZ25}, we have:

\begin{lem}\label{claim 2}
Let $\overline{b_1\cdots b_{4g}}\in \RR$. If $\overline{X}=\overline{b_{2g+1}(b_{2g}\cdots b_2)^{t}X_1}$ ($t\geq 0$) is irreducible, then $\overline{Y}=\overline{b_{2g+1}X_1}$ is also irreducible.
\end{lem}

\begin{proof}
   By the second table in \cite[Lemma 3.2]{WZZ25}, if $\overline{b_{2g+1}X_1}$ is not irreducible, then
   \begin{equation*}
   \overline{X_1}=\left\{\begin{array}{ll}
     \overline{b_1\cdots}    &\mbox{Case (i)}\\
      \overline{b_{2g}\cdots b_{2}b_1\cdots } \mbox{ or }\overline{b_{2g+2}\cdots b_{4g}b_1\cdots }&\mbox{Case (ii)}\\
       \overline{(b_{2g}\cdots b_{2})^rb_1\cdots } \mbox{ or }\overline{(b_{2g+2}\cdots b_{4g})^rb_1\cdots }~(r\geq 2)&\mbox{Case (iii)}\\
      \overline{b_{2g}\cdots b_{2}}~(b_{2g+1}\succ b_{2}) \mbox{ or }\overline{b_{2g+2}\cdots b_{4g}}~(b_{2g+1}\succ b_{4g}) &\mbox{Case (iv)}
   \end{array}\right..   
   \end{equation*}
 For Case (i--iii), it leads to a contradiction that $\overline{b_{2g+1}(b_{2g}\cdots b_2)^{t}X_1}$ is not irreducible. For Case (iv), since $\overline{A}$ is irreducible, we have $b_{2g+1}\prec b_{2}$; or $\overline{A}=\overline{b_{2g+1}(b_{2g}\cdots b_2)^{t}b_{2g+2}\cdots }$ is not irreducible. In conclusion, $\overline{b_{2g+1}X_1}$ is irreducible. 
\end{proof}

\begin{defn}[Reducing-subword pair]
Let $\overline{W_1}, \overline{W_2}\in\W(\gs)$ be two irreducible words. The \emph{reducing-subword pair} of 
$(\overline{W_1},\overline{W_2})$ is defined as
$$\rs(\overline{W_1},\overline{W_2}):=\min\{(\overline{C_1},\overline{C_2})\mid \overline{W_1}=\overline{AC_1}, \overline{W_2}=\overline{C_2B}, ~\nf(W_1W_2)=\overline{AC'B}\},$$
where ``$\min$'' is in the sense of the length-lexicographical order $\prec$ on $\W(\gs)\times\W(\gs)$.  
\end{defn}

\subsection{Conjugator length function}
Finally, we give the formal definition of conjugator length function.
\begin{defn}[Conjugator length function]
    Given a group $G$ with a finite generating set, for any $u\sim v$ in $G$, define $\cl(u,v)$ as follows:
    $$\cl(u,v):=\min\{|w|\mid u=w^{-1}vw, ~w\in G\}.$$
    The \emph{conjugator length function} $\cl:\N\to\N$ is defined so that $\cl(n)$ is the least integer $N$ such that $\cl(u,v)\leq N$ for all conjugate elements $u$ and $v$ with length $|u|+|v|\leq n$.
\end{defn}

\section{Reductions between two irreducible words}\label{sect 3}
In this section, we will discuss reductions of words in the surface group $G=\pi_1(\Sigma_g)$ with the symmetric presentation (\ref{symmetric presentation}). 

\subsection{Reduction procedure}\label{subsect 3.1}
Given two irreducible words $\overline{U}=\overline{u_1\cdots u_m}$ and $\overline{V}=\overline{v_1\cdots v_n}$ with $\overline{UV}$ freely reduced, the reduction "$\overline{UV}\to \nf(UV)$" has finite types. (Refer to the tables in \cite[Appendix B]{WZZ25} for details.) Suppose that $\overline{UV}$ is reducible and the $\llfr$ of $\overline{UV}$ containing $\overline{u_mv_1}$ has length not equal to $(4g-1)$. Then, according to \cite[Section 5]{WZZ25}, we can always reduce $\overline{UV}$ to its normal form by at most three $S_{(i)}$-reductions (see Figure \ref{fig: reduction}):
\begin{figure}
    \centering
    \includegraphics[width=1\linewidth]{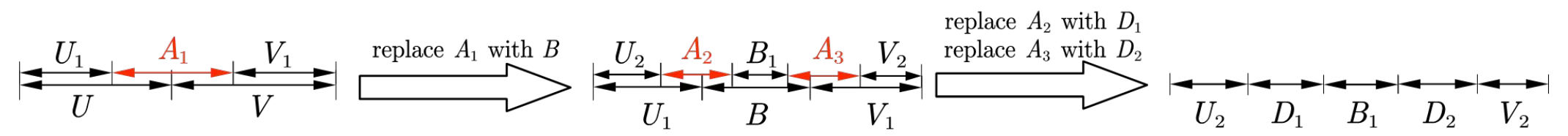}
    \caption{With the settings at the beginning of Subsection \ref{subsect 3.1}, we can always reduce $\overline{UV}$ to its normal form $\nf(UV)$ by at most three $S_{(i)}$-reductions. The processes "replace $\overline{A_2}$ with $\overline{D_1}$" and "replace $\overline{A_3}$ with $\overline{D_2}$" may not exist.}
    \label{fig: reduction}
\end{figure}

\begin{enumerate}
    \item[Step 1.] There must be a reducible subword $\overline{A_1}(\subset \overline{UV})$ of type $S_{(i_1)}$ at the junction of $\overline{UV}$, where $S_{(i_1)}=S_{(2,k)}(2g+1\leq k\leq 4g-2)$, $S_{(3,t_1)}$ or $S_{(4,t_2)}$ with $k,t_2$ maximal. Then, we have the following reduction:
    $$\overline{UV}\xrightarrow{S_{(i_1)}}\overline{U_1BV_1}.$$
    \item[Step 2.] There may exist a reducible subword $\overline{A_2}(\subset \overline{U_1B})$ of type $S_{(i_2)}$ at the junction of $\overline{U_1B}$, where $S_{(i_2)}=S_{(2,2g+1)},S_{(3,t_3)}$ or $S_{(4,t_4)}$ with $t_4$ maximal. Then, we have the following reduction:
    $$\overline{U_1BV_1}\xrightarrow{S_{(i_2)}}\overline{U_2D_1B'V_1}.$$
    \item[Step 3.] There may exist a reducible subword $\overline{A_3}(\subset \overline{B'V_1}\subset\overline{BV_1})$ of type $S_{(i_3)}$ at the junction of $\overline{BV_1}$, where $S_{(i_3)}=S_{(2,2g+1)},S_{(3,t_5)}$ or $S_{(4,t_6)}$ with $t_6$ maximal. Then, we have the following reduction:
$$\overline{U_2D_1B'V_1}\xrightarrow{S_{(i_3)}}\overline{U_2D_1B_1D_2V_1}.$$
\end{enumerate}

All the possible reduction processes of $\overline{UV}$ are listed in \cite[Appendix B, Tables 4--7]{WZZ25}. In these tables, we use the top-left label $*$ to indicate the presence of symmetric cases. At this point, we shall provide an explanation regarding the unlisted symmetric cases, and demonstrate how to derive its symmetric case from a given marked $*$ item. The reason for the occurrence of symmetric cases mainly stem from two aspects:

On one hand, in Step 1, if $S_{(i_1)}=S_{(3,t)}~(t\geq 2), S_{(4,t')}~(t'\geq 1)$, then there are some items (all the items in \cite[Table 6]{WZZ25}), which have symmetric case different from themselves. For example, Item 1 of \cite[Table 6]{WZZ25} has a symmetric case as follows: for some $\overline{b_1\cdots b_{4g}}\in \RR$ and $2\leq r\leq 2g$,
\begin{eqnarray*}
    \overline{C_1C_2}&=& \overline{\underbrace{b_1(b_2\cdots b_{2g})^{t-1}b_2\cdots b_r}_{C_1}\underbrace{b_{r+1}\cdots b_{2g+1}}_{C_2}}\\
    &\xrightarrow{S_{(3,t)}}&\overline{(b_{2g}\cdots b_2)^t}.
\end{eqnarray*}

On the other hand, the different choices of $S_{(i_2)}$ and $S_{(i_3)}$ in Step 2 and Step 3 may yield a symmetric case regardless of whether the Step 1 does or not. For example, Items 2, 3, 4 and 6 in \cite[Table 6]{WZZ25} and all the items with label $*$ in \cite[Table 7]{WZZ25} have symmetric cases, and this is precisely due to this reason. More specifically, we have the following two examples:
\begin{enumerate}
    \item Item 2 in \cite[Table 6]{WZZ25} has a symmetric case as follows: for some $\overline{b_1\cdots b_{4g}}\in\RR$ and $t\geq 2$,
\begin{eqnarray*}
  \overline{C_1C_2}&=& \overline{\underbrace{b_{4g}(b_1\cdots b_{2g-1})^{t-1}b_1\cdots b_r}_{C_1}\underbrace{b_{r+1}\cdots b_{2g}b_2\cdots b_{2g}b_{2g+1}}_{C_2}}\\
    &\xrightarrow{S_{(3,t)}}\xrightarrow{S_{(2,2g+1)}}&\overline{(b_{2g-1}\cdots b_1)^{t-1} b_{2g-1}\cdots b_{2}b_{2g}\cdots b_{2}},
\end{eqnarray*}
 where $S_{(i_1)}=S_{(3,t)}$ while $S_{(i_2)}$ is trivial and $S_{(i_3)}=S_{(2,2g+1)}$.   
    \item Item 4 in \cite[Table 7]{WZZ25} has a symmetric case as follows: for some $\overline{b_1\cdots b_{4g}}\in\RR$ and $2g+1\leq r+s\leq 4g-2$,
\begin{eqnarray*}
  \overline{C_1C_2}&=& \overline{\underbrace{b_1\cdots b_r}_{C_1}\underbrace{b_{r+1}\cdots b_{r+s}b_{r+s-2g+2}\cdots b_{r+s+1}}_{C_2}}\\
    &\xrightarrow{S_{(2,r+s)}}\xrightarrow{S_{(2,2g+1)}}&\overline{b_{2g}\cdots b_{r+s-2g+2}b_{r+s}\cdots b_{r+s-2g+2}},
\end{eqnarray*}
where $S_{(i_1)}=S_{(2,r+s)}$ while $S_{(i_2)}$ is trivial and $S_{(i_3)}=S_{(2,2g+1)}$.
\end{enumerate}

\subsection{Lemmas on reductions between two irreducible words}
By checking the tables in \cite[Tables 4--8]{WZZ25}, we can obtain the following lemmas.

\begin{lem}\label{lem observation 1 of tables}
    Let $\overline{U}=\overline{U_1U_2}$ and $\overline{V}=\overline{V_1V_2}$ be two irreducible words in $\W(\gs)$ with $\overline{UV}$ freely reduced and $\rs(\overline{U},\overline{V})=(\overline{U_2},\overline{V_1})$. If $\nf(U_2V_1)=\overline{X b_1b_2\cdots b_s}$ with $2\leq s\leq 2g$ for some $\overline{b_1\cdots b_{4g}}\in\RR$ and some $\overline{X}\in\W(\gs)$, then $\overline{V_2}\neq \overline{b_{s+1}\cdots}$.

    Symmetrically, if $\nf(U_2V_1)=\overline{b_s\cdots b_2 b_1X}$ for some $\overline{b_1\cdots b_{4g}}\in\RR$, then $\overline{U_1}\neq \overline{\cdots b_{s+1}}$.
\end{lem}
\begin{proof}
For instance, in \cite[Table 4]{WZZ25},  we have  $$\rs(\overline{U},\overline{V})=(\overline{U_2},\overline{V_1})=(\overline{b_1'(b_2'\cdots b_{2g}')^{t_1}},\overline{(b_2'\cdots b_{2g}')^{t_2}b_{2g+1}'}),$$ $$\nf(U_2V_1)=\overline{\underbrace{(b_{2g}'\cdots b_2')^{t_1+t_2-1}}_X b_{2g}'\cdots b_2'}$$ for some $\overline{b_1'\cdots b_{4g}'}\in\RR$. If $\overline{V_2}= \overline{b_{1}'\cdots}$, then $$\overline{V}=\overline{\underbrace{(b_2'\cdots b_{2g}')^{t_2}b_{2g+1}'}_{V_1}\underbrace{b_1'\cdots}_{V_2}}$$ is reducible, contradicting the hypothesis that $\overline{V}$ is irreducible. All other items can be verified by the same way.
\end{proof}

\begin{lem}\label{lem observation of LLFR of reduction}
    Let $\overline{X}=\overline{x_1\cdots x_k b_1\cdots b_r}~(\overline{b_1\cdots b_{4g}}\in\RR, 2\leq r\leq 2g)$ and $\overline{Y}=\overline{y_1\cdots y_n}$ be two irreducible words with $\overline{XY}$ freely reduced but not irreducible. Denote $\rs(\overline{X},\overline{Y}):=(\overline{X_0},\overline{Y_0})$. If $x_k\neq b_{4g}$, then one of the following three items holds:
    \begin{enumerate}
        \item[(a)] $(\overline{X_0}, \overline{Y_0})=(\overline{b_1\cdots b_{2g}},\overline{(b_2\cdots b_{2g})^tb_{2g+1}})(t\geq 1)$ and $r=2g$, which corresponds to \cite[Table 4]{WZZ25}.
        \item[(b)] $(\overline{X_0},\overline{Y_0})=(\overline{\cdots b_1\cdots b_r}, \overline{b_{r+1}\cdots b_{r+s}\cdots })$, which corresponds to \cite[Tables 5--8]{WZZ25}.
        \item[(c)] \begin{equation*}
            (\overline{X_0},\overline{Y_0})=\left\{\begin{array}{l}
                (\overline{b_r},\overline{b_{r-1}\cdots b_{r-2g}}),\\
                (\overline{b_r},\overline{(b_{r-1}\cdots b_{r-2g+1})^tb_{r-2g}})~(t\geq 2),\\
                (\overline{b_r},\overline{(b_{r-1}\cdots b_{r-2g+1})^t})(t\geq 1, b_{r}\succ b_{r-2g+1}),
                \end{array}\right.
        \end{equation*} with the subscript $\cdot_i$ of $b_i$ modulo $4g$, which corresponds to items (2--4) of the second table in \cite[Lemma 3.2]{WZZ25}.
    \end{enumerate}
\end{lem}

\begin{proof}
Note all reduction cases are classified and listed in \cite[Tables 4--8]{WZZ25} according to the length of $\llfr$ at the junction. We adopt the same classification strategy here to present all the reducing-subword pairs $(\overline{X_0},\overline{Y_0}):=\rs(\overline{X}, \overline{Y})$. 

\textbf{Case (1).} There is no $\llfr$ at the junction of $\overline{X}$ and $\overline{Y}$, which corresponds to \cite[Table 4]{WZZ25}. Then
$$(\overline{X_0},\overline{Y_0})=(\overline{b_1(b_2\cdots b_{2g})^{t_1}},\overline{(b_2\cdots b_{2g})^{t_2}b_{2g+1}})$$
for some $t_1,t_2\geq 1$. Combining the hypothesis that $\overline{X}=\overline{x_1\cdots x_k b_1\cdots b_r}$, we obtain $t_1=1$ and hence conclusion (a) holds.

\textbf{Case (2).} There is an $\llfr$, denoted by $\overline{Z}$, at the junction of $\overline{X}$ and $\overline{Y}$, which corresponds to \cite[Tables 5--8]{WZZ25}. Suppose 
\begin{equation*}
   \overline{Z}=\overline{\underbrace{\cdots b_{r}}_{X_1}\underbrace{y_1\cdots y_s}_{Y_1}}, 
\end{equation*}
where $\overline{X}=\overline{\cdots X_1}$ and $\overline{Y}=\overline{Y_1\cdots}$. Then we have two subcases as follows:

\textbf{Subcase (2.1).} $\overline{Z}\subset \overline{b_{i+1}\cdots b_{i+4g}}$, where $\overline{b_{i+1}\cdots b_{i+4g}}$ denotes a cyclic permutation of $\overline{b_1\cdots b_{4g}}$ with the subscripts of $b_{j}$'s $\pmod{4g}$.
Then we have $\overline{y_1\cdots y_s}=\overline{b_{r+1}\cdots b_{r+s}}$ and $y_{s+1}\neq b_{r+s+1}$. 
Since $x_k\neq b_{4g}$, we have $$\overline{Z}=\overline{\underbrace{b_1\cdots b_r}_{X_1}\underbrace{b_{r+1}\cdots b_{r+s}}_{Y_1}}.$$
It follows that conclusion (b) holds. Moreover, according to the length $|\overline{Z}|=r+s$, we have the following correspondences:
\begin{enumerate}
    \item The case of $r+s=2g-1$ corresponds to \cite[Table 5]{WZZ25}.
    \item The case of $r+s=2g$ corresponds to \cite[Table 6]{WZZ25}.
    \item The case of $2g+1\leq r+s\leq 4g-2$ corresponds to \cite[Table 7]{WZZ25}.
    \item The case of $r+s=4g-1$ corresponds to \cite[Table 8]{WZZ25}.
\end{enumerate}

\textbf{Subcase (2.2).} $\overline{Z}\subset \overline{b_{i+4g}\cdots b_{i+1}}$, where $\overline{b_{i+4g}\cdots b_{i+1}}$ denotes a cyclic permutation of $\overline{b_{4g}\cdots b_{1}}$ with the subscripts of $b_{j}$'s $\pmod{4g}$.
Then combining the hypothesis that  $\overline{X}=\overline{x_1\cdots x_kb_1\cdots b_r}$ with $r\geq 2$, we have 
\begin{equation}\label{eq. Z}
 \overline{Z}=\overline{\underbrace{b_r}_{X_1}\underbrace{b_{r-1}\cdots b_{r-s}}_{Y_1}}.   
\end{equation}
By checking all the reducing-subword pairs $(\overline{C_1},\overline{C_2}):=(\overline{d_1\cdots d_{s_1}},\overline{h_1\cdots h_{s_2}})$ in \cite[Tables 5--8]{WZZ25}, we have the following claim: 

\begin{claim}\label{claim 1}
If $\overline{d_{s_1}h_1}\subset \overline{b_{1}\cdots b_{4g}}$ for some $\overline{b_{1}\cdots b_{4g}}\in\RR$, then $\overline{C_1C_2}$ contains no subword $\overline{b_{i+1}b_{i}}$ with subscripts modulo $4g$. Equivalently,
if $\overline{d_{s_1}h_1}\subset \overline{b_{4g}'\cdots b_{1}'}$ for some $\overline{b_{4g}'\cdots b_{1}'}\in\RR$, then $\overline{C_1C_2}$ contains no subword $\overline{b_{i}'b_{i+1}'}$ with subscripts modulo $4g$.
\end{claim}

Now combining $\overline{X}=\overline{x_1\cdots x_k b_1\cdots b_r}$ with Eq. (\ref{eq. Z}) and $\overline{b_{r}b_{r-1}}\subset\overline{b_{4g}\cdots b_1}$, and applying Claim \ref{claim 1} to the reducing-subword pair $\rs(\overline{X},\overline{Y})=(\overline{X_0},\overline{Y_0})$, we have $\overline{X_0}=\overline{b_{r}}$. It then suffices to consider  $\rs(\overline{X_0},\overline{Y})=(\overline{b_r},\overline{Y})$, which corresponds exactly to \cite[Lemma 3.2]{WZZ25}. Thus conclusion (c) holds. 
\end{proof}


\begin{lem}\label{lem observation 2 of tables}
    Let $\overline{U}$ and $\overline{V}$ be two irreducible words in $\W(\gs)$ with $\overline{UV}$ freely reduced. If $\overline{UV}$ has no $(4g-1)$-$\llfr$, then
    $$|U|+|V|-|UV|\leq 4g.$$
\end{lem}
\begin{proof}
Since  $\overline{UV}$ has no $(4g-1)$-$\llfr$, there are at most three reducing process $S_{(2,k)}$($2g+1\leq k\leq 4g-2$), $S_{(3, t)}$ and  $S_{(4, t)}$ according to \cite[Tables 4--7]{WZZ25}. More precisely, at most one $S_{(2,k)}~(k\geq 2g+2)$. Note that  each $S_{(2,k)}$  (resp. $S_{(3, t)}$ )-reduction reduces the word length by $2k-4g$ (resp. $2$), while each $S_{(4, t)}$-reduction does not change the word length. 
Hence, we have
$$|U|+|V|-|UV|  \leq 2(4g-2)-4g+2\cdot 2=4g.$$
\end{proof}

\section{Several lemmas on conjugate reductions}\label{sect 4}

In this section, we present several lemmas concerning conjugate reductions. Let $\overline{X}$ and $\overline{U}$ be two irreducible words in $\W(\gs)$ such that  $\overline{X^{-1}UX}$ is freely reduced. We divide the discussion into two subsections according to whether $\overline{X^{-1}UX}$ contains a $(4g-1)$-$\llfr$.

\subsection{Non-existence of $(4g-1)$-$\llfr$} The following lemma still requires an item-by-item verification of the scenarios in the tables from \cite[Appendix B]{WZZ25} together with some routine arguments. Hence, we present discussions of selected items as its proof.

\begin{lem}\label{lem freely reduced}
    Let $\overline{X}=\overline{x_1\cdots x_{n}}$ and $\overline{U}$ be two irreducible words in $\W(\gs)$ such that  $\overline{X^{-1}UX}$ is freely reduced and has no $(4g-1)$-$\llfr$. If $\overline{\nf(X^{-1}U)X}$ is not freely reduced, then $\rs(\overline{X^{-1}}, \overline{U})=(\overline{X'},\overline{U})$ where $\overline{X'}=\overline{\cdots x_1^{-1}}$ is a subword of $\overline{X^{-1}}$ and $\nf(X^{-1}U)=\overline{Wx_1^{-1}}$ for some irreducible word $\overline{W}$ in $\W(\gs)$. 
    
    Furthermore, one of the following two conclusions holds:
    \begin{enumerate}
    \item[(a)] $\overline{X^{-1}}=\overline{\cdots b_1}$ and $ \overline{U}=\overline{(b_2\cdots b_{2g})^t}$ for some $\overline{b_1\cdots b_{4g}}\in\RR$ and $t\geq 1$.
        \item[(b)] $\overline{Wx_2\cdots x_n}$ is irreducible, and the word length is reduced at most $4g+2$ in the process "$\overline{X^{-1}UX}\to\overline{Wx_2\cdots x_n}=\nf(X^{-1}UX)$".
    \end{enumerate}
\end{lem}
\begin{proof}
Since $\overline{X^{-1}UX}$ is freely reduced whereas $\overline{\nf(X^{-1}U)X}$ is not, we have the reducing-subword pair $\rs(\overline{X^{-1}},\overline{U})=(\overline{X'},\overline{U})$ and $\nf(X^{-1}U)$ must end with $x_1^{-1}$, i.e. $\nf(X^{-1}U)=\overline{Wx_1^{-1}}$ for some $\overline{W}$ in $\W(\gs)$.
Furthermore, if $\overline{Wx_2\cdots x_n}$ is irreducible, then $\nf(X^{-1}UX)=\overline{Wx_2\cdots x_n}$, and hence
\begin{eqnarray*}
& & |\overline{X^{-1}UX}|-|\nf(X^{-1}UX)|\\
    &=& 2|X|+|U|-|Wx_2\cdots x_n|\\
    &=& (|X|+|U|-|X^{-1}U|)+(|X^{-1}U|+|X|-|Wx_2\cdots x_n|)\\
    \mbox{by Lemma \ref{lem observation 2 of tables}}&\leq &4g+(|Wx_1^{-1}|+n-|Wx_2\cdots x_n|)\\
    &=&4g+2.
\end{eqnarray*}
Hence, to prove that conclusion (b) holds, it suffices to prove that $\overline{Wx_2\cdots x_n}$ is irreducible.

We now need to check the items in \cite[Tables 4--7]{WZZ25} by setting $$\rs(\overline{X^{-1}},\overline{U})=(\overline{X'},\overline{U})=(\overline{C_1},\overline{C_2}),$$
$$\nf(X^{-1}U)=\overline{Wx_1^{-1}}=\overline{\cdots\nf(C_1C_2)}.$$
Note that the last letters of both $\overline{C_1}$ and $\overline{\nf(C_1C_2)}$ must be the same, which equal to $x_1^{-1}$.  Therefore, the only item in \cite[Table 4]{WZZ25} does not occur in this lemma. For the items in \cite[Tables 5--7]{WZZ25}, since they are similar and tedious, we only provide a few examples for the sake of illustration, and all other items can be verified by the same way.
    
(1) In \cite[Table 5]{WZZ25}, 
    $$\overline{X'}=\overline{\cdots x_1^{-1}}=\overline{b_1(b_2\cdots b_{2g})^{t_1}b_2\cdots b_{r+1}},\quad \overline{Wx_1^{-1}}=\overline{\cdots(b_{2g}\cdots b_2)^{t_1+t_2+1}}.$$
    Then $b_{r+1}=x_1^{-1}=b_2$ and thus $r=1$. Therefore, 
    \begin{eqnarray*}
\overline{\nf(X^{-1}U)X}&=&\overline{\underbrace{\cdots (b_{2g}\cdots b_2)^{t_1+t_2+1}}_{\nf(X^{-1}U)=Wx_1^{-1}}\underbrace{b_{2g+2}(b_{4g}\cdots b_{2g+2})^{t_1}b_{2g+1}\cdots}_{X=x_1x_2\cdots x_n}}\\      &\xrightarrow{S_{(1)}}&\overline{\underbrace{\cdots (b_{2g}\cdots b_2)^{t_1+t_2}b_{2g}\cdots b_3}_W\underbrace{(b_{4g}\cdots b_{2g+2})^{t_1}b_{2g+1}\cdots}_{x_2\cdots x_n}}=:\overline{H}.
    \end{eqnarray*}
Note that $\overline{W}$ and $\overline{x_2\cdots x_n}$ are irreducible. If $\overline{H}$ is reducible, then $\overline{H}$ must contain a subword of type $S_{(i)}$ at the junction of $\overline{W}$ and $\overline{x_2\cdots x_n}$. However, there is no such subword containing $\overline{b_3b_{4g}}$. Hence $\overline{H}$ is irreducible and conclusion (b) holds.

(2) In item 5 of \cite[Table 6]{WZZ25}, since it has a symmetric reduction, we need to verify two subcases:
\begin{itemize}
        \item $(\overline{X'},\overline{U})=(\overline{b_1\cdots b_r}, ~\overline{b_{r+1}\cdots b_{2g}(b_2\cdots b_{2g})^{t_0-1}})$ and $\overline{Wx_1^{-1}}=\overline{\cdots (b_{2g}\cdots b_2)^{t_0}b_1}$. 
        Then $b_r=x_1^{-1}=b_1$, which implies $r=1$ and hence $\overline{U}=\overline{(b_2\cdots b_{2g})^{t_0}}$. Therefore, conclusion (a) holds.

\item $(\overline{X'},\overline{U})=(\overline{(b_1\cdots 
b_{2g-1})^{t_0-1}b_1\cdots b_r}, ~\overline{b_{r+1}\cdots b_{2g}})$ and $\overline{Wx_1^{-1}}=\overline{\cdots b_{2g}(b_{2g-1}\cdots b_1)^{t_0}}$. Then we also have $b_r=x_1^{-1}=b_1$, which implies $r=1$. Hence conclusion (a) also holds.
\end{itemize}

(3) In item 4 of \cite[Table 7]{WZZ25}, we also have two subcases:
\begin{itemize}
    \item $(\overline{X'},\overline{U})=(\overline{b_{4g}b_1\cdots b_{2g-1}b_1\cdots b_r}, ~\overline{b_{r+1}\cdots b_{r+s}})$ and $\overline{Wx_1^{-1}}=\overline{\cdots b_{2g-1}\cdots b_1b_{2g-1}\cdots b_{r+s-2g+1}}$. Then $b_r=x_1^{-1}=b_{r+s-2g+1}$, which implies that $s=2g-1$ and $2\leq r\leq 2g-1$. Hence, conclusion (a) holds.

    \item $(\overline{X'},\overline{U})=(\overline{b_1\cdots b_r},\overline{b_{r+1}\cdots b_{r+s}b_{r+s-2g+2}\cdots b_{r+s}b_{r+s+1}})$ and $$\overline{Wx_1^{-1}}=\overline{\cdots b_{2g}\cdots b_{r+s-2g+2}b_{r+s}\cdots b_{r+s-2g+2}}.$$
    Then $b_{r}=x_1^{-1}=b_{r+s-2g+2}$, which implies that $s=2g-2$ and $3\leq r\leq 2g$. Therefore, 
    \begin{eqnarray*}
\overline{\nf(X^{-1}U)X}&=&\overline{\underbrace{\cdots b_{2g}\cdots b_{r}b_{r+2g-2}\cdots b_{r+1}b_{r}}_{\nf(X^{-1}U)=Wx_1^{-1}}\underbrace{b_{2g+r}b_{2g+r-1}\cdots b_{2g+1}\cdots }_{X=x_1x_2\cdots x_n}}\\
&\xrightarrow{S_{(1)}}&\overline{\underbrace{\cdots b_{2g}\cdots b_{r}b_{r+2g-2}\cdots b_{r+1}}_W\underbrace{b_{2g+r-1}\cdots b_{2g+1}\cdots}_{x_2\cdots x_n}}=:\overline{H}.
\end{eqnarray*}
Since $3\leq r\leq 2g$, we find that $\overline{H}$ has no subword of type $S_{(i)}$ at the junction of $\overline{W}$ and $\overline{x_2\cdots x_n}$. Therefore, $\overline{H}$ is irreducible and conclusion (b) holds.
\end{itemize}  
\end{proof}

The proof of the following lemma requires writing out all possible reduction processes based on the tables. Here, we provide a complete discussion.

\begin{lem}\label{lem observations only case of 4g-1}
Let $\overline{X}$ and $\overline{U}$ be two irreducible words in $\W(\gs)$ such that $\overline{X^{-1}UX}$ is freely reduced and has no $(4g-1)$-$\llfr$. Then, one of the following items holds:
\begin{enumerate}
    \item[(a)] For some $\overline{b_1\cdots b_{4g}}\in\RR$ and $t_0\geq 1$, we have
    \begin{eqnarray*}
\overline{X^{-1}UX}&=&\left\{\begin{array}{ll}
     \overline{\underbrace{\cdots b_{4g}\cdots b_{2g+2}b_{2g+1}}_{X^{-1}}\underbrace{(b_{2g+2}\cdots b_{4g})^{t_0}b_1b_{2g+3}\cdots b_{4g}}_{U}\underbrace{b_1b_2\cdots b_{2g}\cdots}_{X}} :=\overline{W_1},& \mbox{or} \\
     \overline{\underbrace{\cdots b_2\cdots b_{2g}b_{2g+1}}_{X^{-1}}\underbrace{b_{2g+2}\cdots b_{4g-1}b_{2g+1}(b_{2g+2}\cdots b_{4g})^{t_0}}_{U}\underbrace{b_1b_{4g}\cdots b_{2g+2}\cdots}_{X}}:=\overline{W_2}.& 
\end{array}\right.
    \end{eqnarray*}
    \item[(b)] Neither $\overline{\nf(X^{-1}U)X}$ nor $\overline{X^{-1}\nf(UX)}$ has a $(4g-1)$-$\llfr$.
\end{enumerate}
\end{lem}
\begin{proof}

Denote the reducing-subword pair  
\begin{equation}\label{eq. reducing-pair 1}
    \rs(\overline{X^{-1}},\overline{U})= (\overline{X_0},\overline{U_0}),
\end{equation} where $\overline{X^{-1}}=\overline{X_1X_0}$ and $\overline{U}=\overline{U_0U_1}$. 

If item (b) does not hold, then either $\overline{\nf(X^{-1}U)X}$ or $\overline{X^{-1}\nf(UX)}$ has a $(4g-1)$-$\llfr$.  We shall show $\overline{X^{-1}UX}=\overline{W_1}$ whenever the former holds, and $\overline{X^{-1}UX}=\overline{W_2}$ whenever the latter holds. Since their proofs are almost identical, we only give the proof of the former to avoid repetition. 

Suppose $\overline{\nf(X^{-1}U)X}$ has a $(4g-1)$-$\llfr$. Then $$\rs(\nf(X^{-1}U),\overline{X})= (\overline{C_1},\overline{C_2})$$ as in \cite[Table 8]{WZZ25}. Accordingly, there are two symmetrical cases as follows. 

\textbf{Case (1).} $\nf(X^{-1}U)=\overline{U_2(b_{2g+2}\cdots b_{4g})^t}$ and $\overline{X}=\overline{b_1(b_2\cdots b_{2g})^tX_2}$ for some $t\geq 1$,  $\overline{b_1\cdots b_{4g}}\in\RR$, $\overline{X_2}, \overline{U_2}\in\W(\gs)$ and $b_1\prec b_{2g}$. 

Note that $\overline{U_1X}=\overline{U_1b_1\cdots b_{2g}\cdots}\subset\overline{UX}$ and $\overline{X^{-1}UX}$ is freely reduced and has no $(4g-1)$-$\llfr$, 
we have $\overline{U_1}\neq \overline{\cdots b_{2g+2}\cdots b_{4g}}$. Moreover,  by the discussion at the beginning of Subsection \ref{subsect 3.1}, we know that the reduction 
\begin{eqnarray}\label{eq. reduction X^-1UX}
   \overline{X^{-1}U}\to\nf(X^{-1}U)&=&\overline{X_1\nf(X_0U_0)U_1}\notag\\
   &=&\overline{U_2 (b_{2g+2}\cdots b_{4g})^{t-1}b_{2g+2}\cdots b_{2g+s} \underbrace{b_{2g+s+1}\cdots b_{4g}}_{U_1}} ~(2\leq s\leq 2g)
\end{eqnarray}
can also be finished in at most three steps. (See Figure \ref{fig: reduction}. We use the same notations: $S_{(i_1)}$ (resp. $S_{(i_2)}$, $S_{(i_3)}$) denotes the first (resp. the left, the right) reduction process.)  Then by Lemma \ref{lem observation 1 of tables}, we have $s=2$ or $2g$, and 
\begin{equation}\label{eq. reducing-pair 2}
   \overline{X_1\nf(X_0U_0)}=\left\{\begin{array}{ll}
       \overline{U_2(b_{2g+2}\cdots b_{4g})^{t-1}b_{2g+2}}  & \mbox{if }s=2 , \\
        \overline{U_2(b_{2g+2}\cdots b_{4g})^t} &\mbox{if } s=2g.
    \end{array}\right.
\end{equation}

Note that 
\begin{equation}\label{eq. X-1 case 1}
    \overline{X^{-1}}=\overline{X_2^{-1}(b_{4g}\cdots b_{2g+2})^tb_{2g+1}}
\end{equation}
is irreducible. Then, applying Lemma \ref{lem observation of LLFR of reduction} to $\rs(\overline{X^{-1}},\overline{U})= (\overline{X_0},\overline{U_0})$ in
Eq. (\ref{eq. reducing-pair 1}), 
we obtain
\begin{equation}\label{eq. X0 case 1}
    \overline{X_0}=\left\{\begin{array}{ll}
    \overline{b_{4g}\cdots b_{2g+1}}&\mbox{if Lemma \ref{lem observation of LLFR of reduction}(a) holds},\\
        \overline{\cdots b_{4g}\cdots b_{2g+1}}&\mbox{if Lemma \ref{lem observation of LLFR of reduction}(b) holds},\\
        \overline{b_{2g+1}}   & \mbox{if Lemma \ref{lem observation of LLFR of reduction}(c) holds}.
    \end{array}\right.
\end{equation}
Furthermore, by Lemma \ref{lem observation of LLFR of reduction}, we can obtain all the possible cases of $(\overline{X_0},\overline{U_0})$ and list them in Table \ref{tab: lemma case 1}, where item (1) (resp. items (2--5) and items (6--8)) of Table \ref{tab: lemma case 1} corresponds to Lemma \ref{lem observation of LLFR of reduction}(a) (resp. \ref{lem observation of LLFR of reduction}(b) and \ref{lem observation of LLFR of reduction}(c)).
\begin{table}[ht]
    \centering
    \begin{tabular}{|M{1.9cm}|M{4.5cm}|M{3.6cm}|M{2.3cm}|M{2.4cm}|}
    \hline
    Item&Reducing-subword pair $(\overline{X_0},\overline{U_0})$ & $\nf(X_0U_0)$& Reducing operation $S_{(i_1)},S_{(i_2)},S_{(i_3)}$&Parameter declaration \\
\hline
 $(1);(4.1)$&$(\overline{b_{4g}\cdots b_{2g+1}},$ $\overline{(b_{4g-1}\cdots b_{2g+1})^{t_2}b_{2g}})$ & 
 $\overline{(b_{2g+1}\cdots b_{4g-1})^{1+t_2}}$&
  $S_{(3,t_2+1)},$ $-, -$& $t_2\geq 1$;
  $b_{4g-1}\prec b_{2g}$\\
\hline
\hline
$(2);(7.1,7.2*,$ $7.3*,7.4*)$&$(\overline{\cdots b_{4g}\cdots b_{2g+1}},$ $\overline{b_{2g}\cdots b_{r}b_{2g+r-2}\cdots b_{r-1}})$&
$\overline{\cdots b_{2g+2}\cdots b_{2g+r-2}}$ $\cdot\overline{b_r\cdots b_{2g+r-2}}$&
$S_{(2,4g-r+1)},$ $S_{(i_2)}, S_{(2,2g+1)}$& $3\leq r\leq 2g;$
$b_{2g+r-2}\prec b_{r-1}$\\
\hline
$(3);(7.2,7.5,$ $7.6*,7.7*)$&$(\overline{\cdots b_{4g}\cdots b_{2g+1}},$ $\overline{b_{2g}\cdots b_{r}(b_{2g+r-2}\cdots b_{r})^{t_1}b_{r-1}})$&
$\overline{\cdots b_{2g+2}\cdots b_{2g+r-2}}$ $\cdot\overline{(b_r\cdots b_{2g+r-2})^{t_1}}$&
$S_{(2,4g-r+1)},$ $S_{(i_2)}, S_{(3,t_1)}$& $3\leq r\leq 2g;$
$b_{2g+r-2}\prec b_{r-1}$\\
\hline
$(4);(7.3,7.6,$ $7.8,7.9*)$&$(\overline{\cdots b_{4g}\cdots b_{2g+1}},$ $\overline{b_{2g}\cdots b_{r}(b_{2g+r-2}\cdots b_{r})^{t_1}})$&
$\overline{\cdots b_{2g+2}\cdots b_{2g+r-2}}$ $\cdot\overline{(b_r\cdots b_{2g+r-2})^{t_1}b_{2g+r-1}}$&
$S_{(2,4g-r+1)},$ $S_{(i_2)}, S_{(4,t_1)}$& $3\leq r\leq 2g;$
$b_{2g+r-1}\succ b_{r}$\\
\hline
$(5);(7.4,7.7,$ $7.9,7.10)$&$(\overline{\cdots b_{4g}\cdots b_{2g+1}},$ $\overline{b_{2g}\cdots b_{r}})$&
$\overline{\cdots b_{2g+2}\cdots b_{2g+r-1}}$&
$S_{(2,4g-r+1)},$ $S_{(i_2)}, ~-$&
$3\leq r\leq 2g$\\
\hline
\hline
$(6);(6.1)$&$(\overline{b_{2g+1}},\overline{(b_{2g+2}\cdots b_{4g})^{t_1}b_1})$&
$\overline{(b_{4g}\cdots b_{2g+2})^{t_1}}$&
$S_{(3,t_1)},-,-$& $t_1\geq 2;$
$b_{2g+2}\prec b_1$ \\
\hline
$(7);(6.5)$&$(\overline{b_{2g+1}},\overline{(b_{2g+2}\cdots b_{4g})^{t_1}})$&

$\overline{(b_{4g}\cdots b_{2g+2})^{t_1}b_{2g+1}}$&

$S_{(4,t_1)},-,-$ & $t_1\geq 1$\\
\hline
$(8);(7.10)$&$(\overline{b_{2g+1}},\overline{b_{2g+2}\cdots b_{4g}b_1})$&
$\overline{b_{4g}\cdots b_{2g+2}}$&
$S_{(2,2g+1)},-,-$&
$b_{2g+2}\prec b_1$\\
\hline

\end{tabular}
\caption{All the possible reductions of $(\overline{X^{-1}},\overline{U})$ for \textbf{Case (1)}. Item $(i);(m.n,p.q*)$ means that this is the $i$-th item, corresponding to item $n$ in \cite[Table $m$]{WZZ25} and the symmetric case of item $q$ in \cite[Table $p$]{WZZ25}. ``$S_{(i_j)}=-$" means that it does not exist.}
    \label{tab: lemma case 1}
\end{table}


By checking Table \ref{tab: lemma case 1} item-by-item, we shall show that only items (6)(8) can appear in Lemma \ref{lem observations only case of 4g-1} and conclusion (a) holds. More precisely,
\begin{enumerate}
    \item If items (1)(7) appear, then $\nf(X_0U_0)=\overline{\cdots b_{4g-1}}\mbox{ or }\overline{\cdots b_{2g+1}},$ which contradicts Eq. (\ref{eq. reducing-pair 2}). 
    \item If items (2)(3) appear, then by Eq. (\ref{eq. reducing-pair 2}), we have $r=4,s=2$ and 
    $$\overline{U}=\overline{\underbrace{\cdots b_{3}}_{U_0}\underbrace{b_{2g+3}\cdots b_{4g}}_{U_1}}$$
    is reducible, which contradicts the hypothesis that $\overline{U}$ is irreducible.
    \item If items (4)(5) appear, we can obtain contradictions by the same arguments as in item (2).
    \item If items (6)(8) appear, we have $s=2$, $t=1$ in Eq. (\ref{eq. reducing-pair 2}) and $$\overline{U}=\overline{\underbrace{(b_{2g+2}\cdots b_{4g})^{t_0}b_1}_{U_0} \underbrace{b_{2g+3}\cdots b_{4g}}_{U_1}},\qquad t_0\geq 1, ~b_{2g+2}\prec b_1.$$
    Combining Eq. (\ref{eq. X-1 case 1}), we obtain that $\overline{X^{-1}UX}=\overline{W_1}$ in conclusion (a) holds.
\end{enumerate}

\textbf{Case (2).} $\nf(X^{-1}U)=\overline{U_2(b_{2g+1}\cdots b_{4g-1})^tb_{4g}}$ and $\overline{X}=\overline{(b_1\cdots b_{2g-1})^tX_2}$ with $t\geq 1$ for some $\overline{b_1\cdots b_{4g}}\in\RR$, $\overline{X_2}\in\W(\gs)$ and $b_{2g+1}\prec b_{4g}$. 

By some arguments  analogous to that in Case (1), we have the following equation parallel to Eq. (\ref{eq. reduction X^-1UX}):
$$\nf(X^{-1}U)=\overline{X_1\nf(X_0U_0)U_1}=\overline{U_2 (b_{2g+1}\cdots b_{4g-1})^{t-1}b_{2g+1}\cdots b_{2g+s}\underbrace{b_{2g+s+1}\cdots b_{4g}}_{U_1}} \quad (1\leq s\leq 2g).$$
Then by Lemma \ref{lem observation 1 of tables}, we have $s=1$ or $2g$, and
\begin{equation}\label{eq. s=1,2g}
    \overline{X_1\nf(X_0U_0)}=\left\{\begin{array}{ll}
       \overline{U_2(b_{2g+1}\cdots b_{4g-1})^{t-1}b_{2g+1}}  & \mbox{if }s=1, \\
        \overline{U_2(b_{2g+1}\cdots b_{4g-1})^tb_{4g}} &\mbox{if } s=2g.
    \end{array}\right.
\end{equation}
Note that 
\begin{equation}\label{eq. X-1 case 2}
    \overline{X^{-1}}=\overline{X_2^{-1}(b_{4g-1}\cdots b_{2g+1})^t}
\end{equation}
is irreducible.  Then, applying Lemma \ref{lem observation of LLFR of reduction} to $\rs(\overline{X^{-1}},\overline{U})= (\overline{X_0},\overline{U_0})$ in
Eq. (\ref{eq. reducing-pair 1}), 
we obtain that Lemma \ref{lem observation of LLFR of reduction}(a) does not hold in this case, and
\begin{equation}\label{eq. X0}
    \overline{X_0}=\left\{\begin{array}{ll}
        \overline{\cdots b_{4g-1}\cdots b_{2g+1}}&\mbox{if Lemma \ref{lem observation of LLFR of reduction}(b) holds},\\
        \overline{b_{2g+1}}   & \mbox{if Lemma \ref{lem observation of LLFR of reduction}(c) holds}.
    \end{array}\right.
\end{equation}
Furthermore, by Lemma \ref{lem observation of LLFR of reduction}, we can obtain all the possible cases of $(\overline{X_0},\overline{U_0})$ and list them in Table \ref{tab: lemma case 2}, where items (1--7) (resp. items (8--9)) of Table \ref{tab: lemma case 2} correspond to Lemma \ref{lem observation of LLFR of reduction}(b) (resp. \ref{lem observation of LLFR of reduction}(c)).
\begin{table}[h]
    \centering
    \begin{tabular}{|M{1.9cm}|M{4.4cm}|M{3.6cm}|M{2.5cm}|M{2.4cm}|}
    \hline
    Item &Reducing-subword pair $(\overline{X_0},\overline{U_0})$ & $\nf(X_0U_0)$& Reducing operation $S_{(i_1)},S_{(i_2)}, ~S_{(i_3)}$&Parameter declaration \\
\hline
$(1);(6.1,6.2,$ $6.3,6.4)$& $(\overline{\cdots b_{4g-1}\cdots b_{2g+1}},$ $\overline{b_{2g}(b_{4g-2}\cdots b_{2g})^{t_1}b_{2g-1}})$&

$\overline{\cdots b_{2g-1}\cdots b_{4g-2}}$ $\overline{\cdot(b_{2g}\cdots b_{4g-2})^{t_1}}$&

$S_{(3,t_1+1)}, S_{(i_2)}, -$&
$t_1\geq 1;$
$b_{4g-2}\prec b_{2g-1}$\\

\hline
$(2);(6.5*)$& $(\overline{(b_{4g-1}\cdots b_{2g+1})^{t'}},\overline{b_{2g}})$&

$\overline{b_{2g}(b_{2g+1}\cdots b_{4g-1})^{t'}}$&

$S_{(4,t')},$ $-, -$&
$t'\geq 1;$
$b_{4g-1}\succ b_{2g}$\\
\hline
$(3);(6.6)$& $(\overline{(b_{4g-1}\cdots b_{2g+1})^{t'}},$ $\overline{b_{2g}(b_{4g-2}\cdots b_{2g})^{t_1}})$&

$\overline{b_{2g}(b_{2g+1}\cdots b_{4g-1})^{t'-1}}$ $\cdot\overline{b_{2g+1}\cdots b_{4g-2}}$ 
$\overline{\cdot(b_{2g}\cdots b_{4g-2})^{t_1}b_{4g-1}}$&

$S_{(4,t')},$ $-, ~S_{(4,t_1)}$&
$t'\geq t\geq 1;$
$b_{4g-1}\succ b_{2g}$\\
\hline
$(4);(7.1,$ $7.2*,7.3*,$ $7.4*)$& $(\overline{\cdots b_{4g-1}\cdots b_{2g+1}},$ $\overline{b_{2g}\cdots b_{r}b_{2g+r-2}\cdots b_rb_{r-1}})$&

$\overline{\cdots b_{2g+1}\cdots b_{2g+r-2}}$ $\overline{\cdot b_{r}\cdots b_{2g+r-2}}$&

$S_{(2,4g-r)},$ $S_{(i_2)}, S_{(2, 2g+1)}$&
 $2\leq r\leq 2g-1;$ $b_{2g+r-2}\prec b_{r-1}$\\
\hline
$(5);(7.2*,7.5,$ $7.6*,7.7*)$& $(\overline{\cdots b_{4g-1}\cdots b_{2g+1}},$ $\overline{b_{2g}\cdots b_{r}(b_{2g+s-2}\cdots b_r)^{t_1}b_{r-1}})$&

$\overline{\cdots b_{2g+1}\cdots b_{2g+r-2}}$ $\cdot\overline{(b_{r}\cdots b_{2g+r-2})^{t_1}}$&

$S_{(2,4g-r)},$ $S_{(i_2)}, S_{(3,t_1)}$&
$t_1\geq 2;$ $2\leq r\leq 2g-1;$ 
$b_{2g+r-2}\prec b_{r-1}$\\
\hline
$(6);(7.3,7.6,$ $7.8,7.9*)$& $(\overline{\cdots b_{4g-1}\cdots b_{2g+1}},$ $\overline{b_{2g}\cdots b_{r}(b_{2g+r-2}\cdots b_r)^{t_1}})$&

$\overline{\cdots b_{2g+1}\cdots b_{2g+r-2}}$ $\cdot\overline{(b_{r}\cdots b_{2g+r-2})^{t_1}b_{2g+r-1}}$&

$S_{(2,4g-r)},$ $S_{(i_2)}, ~S_{(4,t_1)}$&
$2\leq r\leq 2g-1;$
$b_{2g+r-1}\succ b_{r}$\\
\hline
$(7);(7.4,7.7,$ $7.9,7.10)$& $(\overline{\cdots b_{4g-1}\cdots b_{2g+1}},$ $\overline{b_{2g}\cdots b_{r}})$&

$\overline{\cdots b_{2g+1}\cdots b_{2g+r-1}}$&

$S_{(2,4g-r)},$ $S_{(i_2)}, ~-$ &

$2\leq r\leq 2g-1$\\
\hline
\hline
$(8);(7.10)$ &$(\overline{b_{2g+1}}, ~\overline{b_{2g+2}\cdots b_{4g}b_1})$&
$\overline{b_{4g}\cdots b_{2g+2}}$&
$S_{(2,2g+1)}, -, -$&
$b_{2g+2}\prec b_1$\\
\hline
$(9);(6.1)$&$(\overline{b_{2g+1}}, ~\overline{(b_{2g+2}\cdots b_{4g})^{t_1}b_1})$&
$\overline{(b_{4g}\cdots b_{2g+2})^{t_1}}$&
$\quad S_{(3,t_1)}, -,-$& $t_1\geq 2;$
$b_{2g+2}\prec b_1$\\
\hline
    \end{tabular}
    \caption{All the possible reductions of $(\overline{X^{-1}},\overline{U})$ for \textbf{Case (2)}. Item $(i);(m.n,p.q*)$ means that this is the $i$-th item, corresponding to Item $n$ in \cite[Table $m$]{WZZ25} and the symmetric case of Item $q$ in \cite[Table $p$]{WZZ25}. ``$S_{(i_j)}=-$" means that it does not exist.} 
    \label{tab: lemma case 2}
\end{table}

We now show that Table \ref{tab: lemma case 2} can not appear in Lemma \ref{lem observations only case of 4g-1} by checking it item-by-item.

\begin{enumerate}
 \item If items (1--3)(8--9) appear, then $\nf(X_0U_0)=\overline{\cdots b_{4g-2}},~\overline{\cdots b_{4g-1}}\mbox{ or }\overline{\cdots b_{2g+2}},$ contradicting Eq. (\ref{eq. s=1,2g}).
\item If items (4--5) appear, then combining Eq. (\ref{eq. s=1,2g}), we have $r=3,s=1$ and the irreducible word
    $$\overline{U}=\overline{\underbrace{\cdots b_2}_{U_0}\underbrace{b_{2g+2}\cdots b_{4g}}_{U_1}}$$
    is reducible, which is contradictory. 
\item If items (6--7) appear, then we have the same contradiction as that of item (4).
\end{enumerate}

In conclusion, we have finished the discussion of Case (1) and Case (2) and thus also finish the proof.
\end{proof}

\subsection{Existence of $(4g-1)$-$\llfr$} To handle the case that ``$\overline{X^{-1}UX}$ contains $(4g-1)$-$\llfr$s'', we have the following three lemmas.


\begin{lem}\label{lem 4g-1-reduction 1}
    Let $\overline{X}$ and $\overline{U}$ be two irreducible words such that $\overline{X^{-1}UX}$ is freely reduced. If $\overline{U}$ is cyclically irreducible and $\overline{U}\neq \overline{b_1\cdots b_{2g-1}}$ for any $\overline{b_1\cdots b_{4g}}\in\RR$, then every $\llfr$ in $\overline{X^{-1}UX}$ either starts in $\overline{X^{-1}}$ and ends in $\overline{U}$, or symmetrically, starts in $\overline{U}$ and ends in $\overline{X}$. Moreover, there is at most one $(4g-1)$-$\llfr$ in $\overline{X^{-1}UX}$.
\end{lem}
\begin{proof}
Let $\overline{X}=\overline{x_1\cdots x_n}$. If $\overline{X^{-1}UX}$ contains an $\llfr$ starting in $\overline{X^{-1}}$ and ending in $\overline{X}$, then $\overline{x_1^{-1}Ux_1}$ is also a fractional relator. This implies that $\overline{U}=\overline{b_1\cdots b_r}$, $x^{-1}=b_{4g}$ and $x_1=b_{r+1}$ for some $\overline{b_1\cdots b_{4g}}\in\RR$ and $1\leq r\leq 2g$. 
It follows that $b_{r+1}=b_{4g}^{-1}=b_{2g}$ and hence $r=2g-1$, contradicting the hypothesis $\overline{U}\neq\overline{b_1\cdots b_{2g-1}}$. Hence, every $\llfr$ in $\overline{X^{-1}UX}$ either starts in $\overline{X^{-1}}$ and ends in $\overline{U}$, or symmetrically, starts in $\overline{U}$ and ends in $\overline{X}$.

We now assume that there exist two $(4g-1)$-$\llfr$s in $\overline{X^{-1}UX}$, denoted by $\overline{L_1}$ and $\overline{L_2}$, such that $\overline{L_1}\subset\overline{X^{-1}U}$ and $\overline{L_2}\subset\overline{UX}$. Since both $\overline{X}$ and $\overline{U}$ are irreducible, the length of the common subword of $\overline{L_2}$ and $\overline{U}$ is either $2g-1$ or $2g$. Accordingly, there are two possible cases for $\overline{L_2}$: 

\textbf{Case (1).} $\overline{UX}=\overline{\underbrace{\cdots b_{2g+2}\cdots b_{4g}}_{U}\underbrace{b_1\cdots b_{2g}\cdots }_{X}}$ for some $\overline{b_1\cdots b_{4g}}\in\RR$. Then $\overline{X^{-1}}=\overline{\cdots b_{4g}\cdots b_{2g+1}}$. Since $\overline{L_1}\subset\overline{X^{-1}U}$, we have $$\overline{X^{-1}UX}=\overline{\underbrace{\cdots b_{4g}\cdots b_{2g+1}}_{X^{-1}}\underbrace{b_{2g}\cdots b_1\cdots}_{U}X}.$$
Therefore, $\overline{U}=\overline{b_{2g}\cdots b_{4g}}$ contradicting the hypothesis that $\overline{U}$ is cyclically irreducible.
    
    \textbf{Case (2).} $\overline{UX}=\overline{\underbrace{\cdots b_{2g+1}\cdots b_{4g}}_{U}\underbrace{b_1\cdots b_{2g-1}\cdots}_{X} }$ for some $\overline{b_1\cdots b_{4g}}\in\RR$. Then we can obtain a contradiction by an argument analogous to that in Case (1).
 
Hence, there are at most one $(4g-1)$-$\llfr$ in $\overline{X^{-1}UX}$.
\end{proof}

\begin{lem}\label{lem 4g-1-reduction 2}
  Let $\overline{X}$ and $\overline{U}$ be two irreducible words such that $\overline{X^{-1}UX}$ is freely reduced. Suppose $\overline{U}$ is cyclically irreducible and $\overline{U}\neq\overline{(b_1\cdots b_{2g-1})^{t}}$ for any $\overline{b_1\cdots b_{4g}}\in\RR$ and $t\geq 1$. If there is a $(4g-1)$-$\llfr$ in $\overline{X^{-1}UX}$, then there exists two irreducible words $\overline{Y}$ and $\overline{V}$ such that the followings hold:
  \begin{enumerate}
      \item [(a)] $\overline{V}$ is a cyclic permutation of $\overline{U}$.
      \item [(b)] $X^{-1}UX=Y^{-1}VY$ as group elements.
      \item [(c)] $\overline{Y^{-1}VY}$ is freely reduced and has no $(4g-1)$-$\llfr$.
  \end{enumerate}
\end{lem}
\begin{proof}
According to Lemma \ref{lem 4g-1-reduction 1}, $\overline{X^{-1}UX}$ has exactly one $(4g-1)$-$\llfr$. Without loss of generality, we suppose it is at the junction of $\overline{X^{-1}}$ and $\overline{U}$. Then, according to \cite[Table 8]{WZZ25}, we have the following two cases:

\textbf{Case (1).}  $\overline{X^{-1}}=\overline{X_1^{-1}(b_{2g+2}\cdots b_{4g})^{t_0}}$ and $\overline{UX}=\overline{b_1(b_2\cdots b_{2g})^{t_0}Z}$ with $t_0\geq 1$ maximal and $b_1\prec b_{2g}$ for some irreducible words $\overline{X_1}$ and $\overline{Z}$. Then the irreducible word
\begin{equation}\label{eq. X for case 1}
   \overline{X}=\overline{(b_{2g}\cdots b_2)^{t_0}X_1} 
\end{equation}
for $\overline{X_1}\neq \overline{(b_{2g}\cdots b_2)^{t'}b_{1}\cdots} ~(t'\geq 0)$. We have two subcases:

\textbf{Subcase (1.1).} If $\overline{X}\nsubseteq\overline{Z}$, then we have 
\begin{equation}\label{eq U for subcase 1.1}
    \overline{UX}=\overline{\underbrace{b_1(b_2\cdots b_{2g})^{{t_0}-1}b_2\cdots b_{2g-1}}_{U}\underbrace{b_{2g}Z}_X}.
\end{equation}
Since $\overline{U}\neq \overline{b_1\cdots b_{2g-1}}$, we have $t_0\geq 2$. Let 
\begin{equation}\label{eq. V Y in 1.1}
    \overline{V}:=\overline{b_2\cdots b_{2g-1}b_1(b_2\cdots b_{2g})^{t_0-1}}, \qquad \overline{Y}:=\nf(b_{2g+1}b_{2g}\cdots b_2X_1).
\end{equation} 
Since $\overline{b_{2g}\cdots b_2X_1}\subset \overline{X}=\overline{b_{2g}Z}$ is irreducible, by \cite[Lemma 3.2]{WZZ25}, we have two cases for $\overline{Y}$:

   (1) If $b_{2g+1}\succ b_2$, then the irreducible word $\overline{X_1}$ can be written as
    $\overline{X_1}=\overline{(b_{2g}\cdots b_2)^{t_1-1}X_2}$ with $t_1\geq 1$ maximal for some irreducible word $\overline{X_2}\neq\overline{b_1\cdots}$. Accordingly,  by item (4) of the second table in \cite[Lemma 3.2]{WZZ25}, we have
    $$\overline{Y}=\nf(b_{2g+1}b_{2g}\cdots b_2X_1)=\overline{(b_2\cdots b_{2g})^{t_1}b_{2g+1}X_2}.$$
    
    (2) If $b_{2g+1}\prec b_2$, then by item (5) of the second table in \cite[Lemma 3.2]{WZZ25}, we have
    $$\overline{Y}=\nf(b_{2g+1}b_{2g}\cdots b_2X_1)=\overline{b_{2g+1}b_{2g}\cdots b_2X_1}.$$

Then, by Eqs. (\ref{eq. X for case 1})(\ref{eq U for subcase 1.1})(\ref{eq. V Y in 1.1}), we can verify that $\overline{V}$ and $\overline{Y}$ satisfy the required conditions (a--c). 

    \textbf{Subcase (1.2).} If $\overline{X}\subseteq\overline{Z}$, then the irreducible word
    \begin{equation}\label{eq U for subcase 1.2}
        \overline{U}=\overline{b_1(b_2\cdots b_{2g})^{t_0}U_1}.
    \end{equation}
    Denote 
    \begin{equation}\label{eq Y V  in subcase 1.2}
        \overline{V}:=\overline{U_1b_1(b_2\cdots b_{2g})^{t_0}}, \qquad \overline{Y}:=\nf(b_{2g+1}X_1).
    \end{equation} 
See Figure \ref{fig: Case 1.2 for 4g-1} for this process.
\begin{figure}
    \centering
    \includegraphics[width=0.75\linewidth]{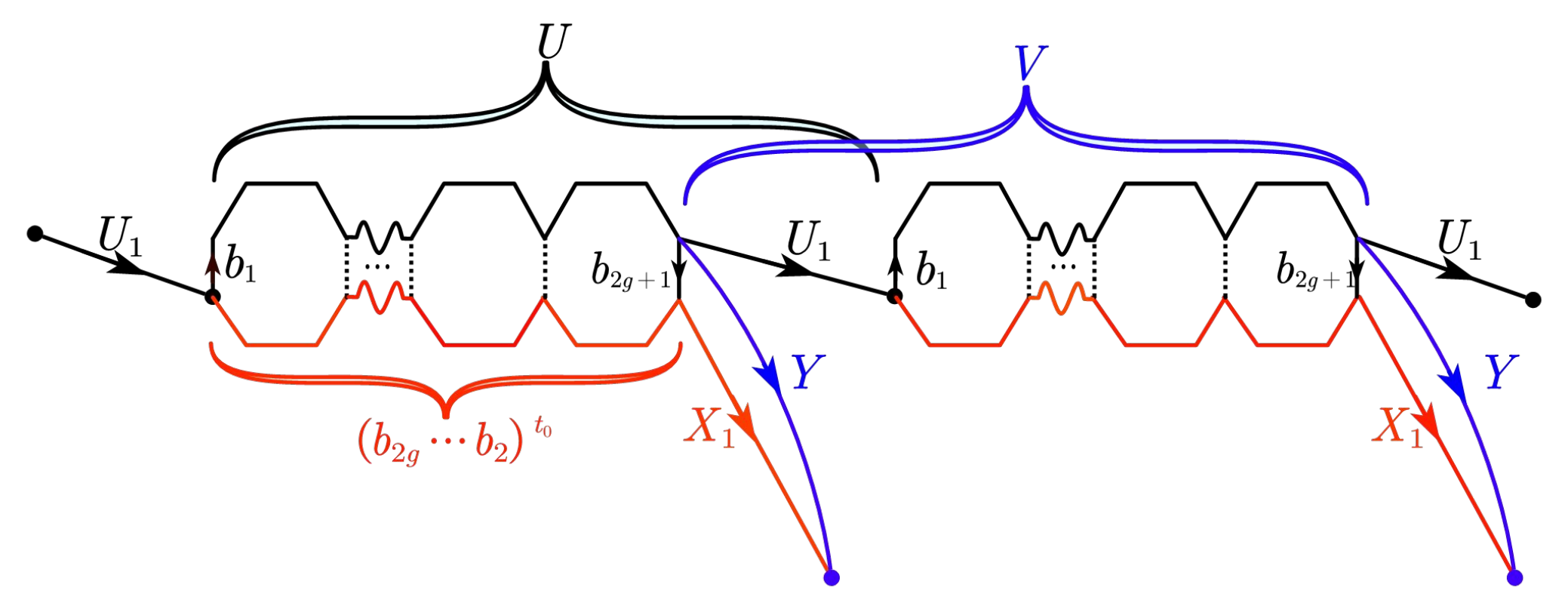}
    \caption{In Subcase (1.2), $\overline{X^{-1}}=\overline{X_1^{-1}(b_{2g+2}\cdots b_{4g})^{t_0}}$ and $\overline{U}=\overline{b_1(b_{2}\cdots b_{2g})^{t_0} U_1}$. Observe that the two paths corresponding to $\overline{X^{-1}UX}$ and $\overline{Y^{-1}VY}$ share common endpoints in $\mathbf{H}^2$. Hence,  as group elements, $X^{-1}UX=Y^{-1}VY$ in the surface group $G$ with the symmetric presentation.}
    \label{fig: Case 1.2 for 4g-1}
\end{figure} 
Since $\overline{X_1}\neq \overline{(b_{2g}\cdots b_2)^{t'}b_{1}\cdots} (t'\geq 0)$ and $\overline{b_{2g}\cdots b_2X_1}$ is irreducible, by \cite[Lemma 3.2]{WZZ25}, we have two cases for $\overline{Y}$:

   (1) If $b_{2g+1}\succ b_2 \mbox{ and }\overline{X_1}=\overline{(b_{2g}\cdots b_2)^{t_1}X_2}$ with $t_1\geq 1$ maximal, then \begin{equation}\label{eq Y for subcase 1.2.1}
        \overline{Y}=\overline{(b_2\cdots b_{2g})^{t_1}b_{2g+1}X_2}.
    \end{equation}
    Now combining Eq. (\ref{eq Y V  in subcase 1.2}) with Eq. (\ref{eq Y for subcase 1.2.1}), we have $$\overline{Y^{-1}VY}=\overline{\underbrace{X_2^{-1}b_1(b_{4g}\cdots b_{2g+2})}_{Y^{-1}}\underbrace{U_1b_1(b_2\cdots b_{2g})^{t_0}}_{V}\underbrace{(b_2\cdots b_{2g})^{t_1}b_{2g+1}X_2}_Y}.$$
    Since $\overline{U}=\overline{b_1(b_2\cdots b_{2g})^{t_0}U_1}$ is irreducible, we have $\overline{U_1}\neq \overline{b_{2g+1}\cdots }$. Then, by Eqs. (\ref{eq. X for case 1})(\ref{eq U for subcase 1.2}) (\ref{eq Y V  in subcase 1.2})(\ref{eq Y for subcase 1.2.1}), we can directly verify that these $\overline{V}$ and $\overline{Y}$ satisfy the required conditions (a--c).
    
   (2) If $b_{2g+1}\prec b_2$ or $\overline{X_1}\neq \overline{b_{2g}\cdots b_2\cdots}$, then by \cite[Lemma 3.2]{WZZ25},
    \begin{equation}\label{eq Y for subcase 1.2.2}
        \overline{Y}=\overline{b_{2g+1}X_1},
    \end{equation}
    and $$\overline{Y^{-1}VY}=\overline{\underbrace{X_1^{-1}b_1}_{Y^{-1}}\underbrace{U_1b_1(b_2\cdots b_{2g})^{t_0}}_{V}\underbrace{b_{2g+1}X_1}_Y}.$$
  Furthermore, by Eqs. (\ref{eq. X for case 1})(\ref{eq U for subcase 1.2}), we have $\overline{X_1}\neq \overline{b_{2g+2}\cdots}$ and $\overline{U_1}\neq \overline{b_{4g}\cdots }$.
  Since $t_0$ is maximal, we have
  $$\overline{X_1^{-1}}=\overline{\cdots b_{2g+2}\cdots b_{4g}}\quad \mbox{and} \quad \overline{U_1}=\overline{b_2\cdots b_{2g}\cdots }$$ can not hold simultaneously.
  Then, combining Eqs. (\ref{eq. X for case 1})(\ref{eq U for subcase 1.2})(\ref{eq Y V  in subcase 1.2})(\ref{eq Y for subcase 1.2.2}), we can directly verify that $\overline{V}$ and $\overline{Y}$ satisfy the required conditions (a--c). 
    
Therefore, we have finished the discussion for Case (1).\\

\textbf{Case (2).} $\overline{X^{-1}}=\overline{X_1^{-1}(b_{2g+2}\cdots b_{4g})^{t_0}b_1}$ and $\overline{UX}=\overline{(b_{2}\cdots b_{2g})^{t_0} Z}$ with $t_0\geq 1$ maximal and $b_{2g+2}\prec b_1$ for some irreducible words $\overline{X_1}$ and $\overline{Z}$. Then the irreducible words
\begin{equation}\label{eq X U for case 2}
    \overline{X}=\overline{b_{2g+1}(b_{2g}\cdots b_2)^{t_0}X_1},\qquad \overline{U}=\overline{(b_2\cdots b_{2g})^{t_0}U_1}.
\end{equation}
By the hypothesis  $\overline{U}\neq \overline{(b_1'\cdots b_{2g-1}')^t}$ for any $\overline{b'_1\cdots b'_{4g}}\in\RR$ and any $t\geq 1$, we have $\overline{U_1}\neq \overline{(b_{2}\cdots b_{2g})^{t'}}$ for any $t'\geq 0$. Denote
\begin{equation}\label{eq Y V for case 2}
    \overline{V}:=\overline{U_1(b_2\cdots b_{2g})^{t_0}}, \qquad \overline{Y}:=\overline{b_{2g+1}X_1}.
\end{equation}
See Figure \ref{fig: Case 2 for 4g-1} for this process.
\begin{figure}
    \centering
    \includegraphics[width=0.75\linewidth]{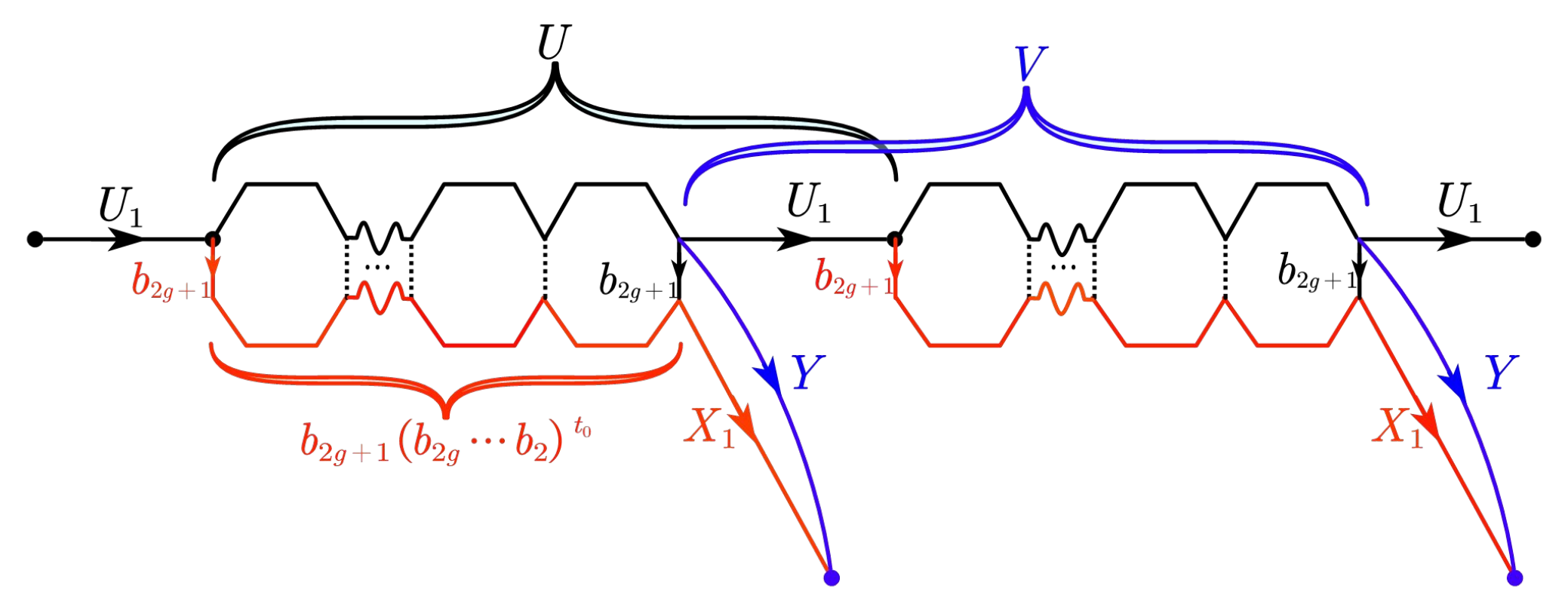}
    \caption{In Case (2), $\overline{X^{-1}}=\overline{X_1^{-1}(b_{2g+2}\cdots b_{4g})^{t_0}b_1}$ and $\overline{U}=\overline{(b_{2}\cdots b_{2g})^{t_0} U_1}$. Observe that the two paths corresponding to $\overline{X^{-1}UX}$ and $\overline{Y^{-1}VY}$ share common endpoints in $\mathbf{H}^2$. Hence, as group elements, $X^{-1}UX=Y^{-1}VY$ in the surface group $G$ with the symmetric presentation.}
    \label{fig: Case 2 for 4g-1}
\end{figure} 
Since $\overline{U}$ is cyclically irreducible, we obtain the irreducibility of $\overline{V}$. Furthermore,  by Lemma \ref{claim 2}, $\overline{Y}$ is also irreducible. 
We now show that

$$\overline{Y^{-1}VY}=\overline{\underbrace{X_1^{-1}b_1}_{Y^{-1}}\underbrace{U_1(b_2\cdots b_{2g})^{t_0}}_{V}\underbrace{b_{2g+1}X_1}_{Y}}$$
satisfies conclusion (c) in the following. Indeed,

(1) If $\overline{Y^{-1}VY}$ is not freely reduced, then $\overline{U_1}=\overline{b_{2g+1}\cdots}$ and thus $\overline{U}=\overline{(b_2\cdots b_{2g})^{t_0}b_{2g+1}\cdots}$. Since $b_{2g+2}\prec b_1$, we have $b_2\succ b_{2g+1}$ and thus $\overline{U}$ is reducible by $S_{(4,t_0)}$, which is contradictory.

 (2) If $\overline{Y^{-1}VY}$ contains a $(4g-1)$-$\llfr$ (denoted by $\overline{K}$), then by Lemma \ref{lem 4g-1-reduction 1}, $\overline{K}$ either starts in $\overline{Y^{-1}}$ and ends in $\overline{V}$ or starts in $\overline{V}$ and ends in $\overline{Y}$. If $\overline{K}$ starts in $\overline{V}$ and ends in $\overline{Y}$, then $\overline{X_1}=\overline{b_{2g+2}\cdots }$ and thus by Eq. (\ref{eq X U for case 2}), the irreducible word
 $$\overline{X}=\overline{b_{2g+1}(b_{2g}\cdots b_2)^{t_0}b_{2g+2}\cdots }$$ is reducible, which is contradictory. If $\overline{K}$ starts in $\overline{Y^{-1}}$ and ends in $\overline{V}$, then by Eq. (\ref{eq X U for case 2}), we have $\overline{U_1}\neq \overline{b_{4g}\cdots }$. Furthermore, since both $\overline{Y^{-1}}$ and $\overline{V}$ are irreducible, the length of the common subword of $\overline{K}$ and $\overline{V}$ is either $2g-1$ or $2g$. Accordingly, there are two possible cases:
  
\begin{enumerate}
        \item [(i)] $\overline{Y^{-1}}=\overline{\cdots b_{2g+3}\cdots b_{4g}b_1},\quad \overline{V}=\overline{b_2\cdots b_{2g}b_{2g+1}\cdots}$.
        
        \item [(ii)]  $\overline{Y^{-1}}=\overline{\underbrace{\cdots b_{2g+2}\cdots b_{4g}}_{X_1^{-1}}b_1},\quad \overline{V}=\overline{\underbrace{b_2\cdots b_{2g}\cdots }_{U_1}(b_2\cdots b_{2g})^{t_0}}$.
\end{enumerate}
However, as demonstrated below, both cases lead to a contradiction: by Eqs. (\ref{eq X U for case 2})(\ref{eq Y V for case 2}), in Case (i), the cyclically irreducible word $\overline{U}=\overline{(b_2\cdots b_{2g})^{t_0}\underbrace{b_2\cdots b_{2g}b_{2g+1}\cdots}_{U_1}}$ is reducible by $S_{(4,t_0+1)}$; in Case (ii), since $\overline{U}\neq \overline{(b_2\cdots b_{2g})^t}$, we have $\overline{U_1}=\overline{b_2\cdots b_{2g}\cdots }$. Hence,
$$\overline{X^{-1}U}=\overline{\underbrace{\underbrace{\cdots b_{2g+2 }\cdots b_{4g}}_{X_1^{-1}}(b_{2g+2}\cdots b_{4g})^{t_0}b_1}_{X^{-1}}~\underbrace{(b_{2}\cdots b_{2g})^{t_0}\underbrace{b_2\cdots b_{2g}\cdots}_{U_1}}_{U}}\quad ,$$ which contradicts the maximality of $t_0$.

According the above discussion, we can directly verify that $\overline{V}$ and $\overline{Y}$ in Eq. (\ref{eq Y V for case 2}) satisfy the required conditions (a--c). This completes the analysis of Case (2) and hence the proof of Lemma \ref{lem 4g-1-reduction 2}.    
\end{proof}

For the special case that $\overline{U}=\overline{(b_1\cdots b_{2g-1})^t}$, we have the following lemmas.

\begin{lem}\label{lem used in final}
 Let $\overline{U}=\overline{(b_1\cdots b_{2g-1})^t}$ for some $\overline{b_1\cdots b_{4g}}\in\RR$ and $t\geq 1$, and let $\overline{X}\neq \overline{b_{2g}\cdots}$ be an irreducible word. If $\overline{X^{-1}UX}$ is freely reduced. then it is irreducible.
\end{lem}

 \begin{proof}
 Since $\overline{X}\neq \overline{b_{2g}\cdots}$, i.e., $\overline{X^{-1}}\neq \overline{\cdots b_{4g}}$, there is no $\llfr$ at the junction of $\overline{X^{-1}U}$ or at the junction of $\overline{UX}$. Furthermore, note that $(\overline{X^{-1}},\overline{U})$ has no reducing-subword pair satisfying \cite[Table 4]{WZZ25}, we have $\overline{X^{-1}U}$ is irreducible, i.e., $$\nf(X^{-1}U)=\overline{X^{-1}U}.$$ We show that $\overline{\nf(X^{-1}U)X}$ is also irreducible. Indeed, if $\overline{\nf(X^{-1}U)X}$ is reducible, then the reducing-subword pair $\rs(\overline{\nf(X^{-1}U)},\overline{X})$ would satisfy \cite[Table 4]{WZZ25}, since no $\llfr$ occurs at the junction of $\overline{\nf(X^{-1}U)X}$. Then, we have 
$$\overline{\nf(X^{-1}U)}=\overline{\underbrace{\cdots b_{4g}(b_1\cdots b_{2g-1})^{t_1}}_{X^{-1}}\underbrace{(b_{1}\cdots b_{2g-1})^t}_{U}}~,\qquad \overline{X}=\overline{(b_1\cdots b_{2g-1})^{t_2}b_{2g}\cdots}$$
for some $t_1\geq 0$ and $t_2\geq 1$. It follows that $\overline{X^{-1}}$ and $\overline{X}$ are contradictory. Hence, $\overline{X^{-1}UX}$ is irreducible.
 \end{proof}

\begin{lem}\label{lem for special case to cancel 4g-1}
    Let $\overline{U}=\overline{(b_1\cdots b_{2g-1})^t}$ for some $\overline{b_1\cdots b_{4g}}\in\RR$ and $t\geq 1$, and let $\overline{X}$ be an irreducible word such that $\overline{X^{-1}UX}$ is freely reduced. If $\overline{X^{-1}U}$ or $\overline{UX}$ contains a $(4g-1)$-$\llfr$, then there exists an irreducible word $\overline{Y}$ such that
    \begin{enumerate}
      \item [(a)] $X^{-1}UX=Y^{-1}UY$ as group elements.
      \item [(b)] $\overline{Y^{-1}UY}$ is freely reduced.
      \item [(c)] Neither $\overline{Y^{-1}U}$ nor $\overline{UY}$ has a $(4g-1)$-$\llfr$.
  \end{enumerate}
\end{lem}
\begin{proof}
Suppose $\overline{UX}$ contains a $(4g-1)$-$\llfr$. Then we have
$$\overline{X}=\overline{b_{2g}(b_{2g+1}\cdots b_{4g-1})^{t_1}X_1}$$ for some irreducible word $\overline{X_1}\in\W(\gs)$ with $t_1\geq 1$ maximal. Let $\overline{Y}:=\overline{b_{2g}X_1}$. Then by Lemma \ref{claim 2}, $\overline{Y}$ is irreducible. Moreover, we can directly verify that $\overline{Y}$ satisfies conclusions (a) and (b). We now prove conclusion (c) in the following.

If $\overline{Y^{-1}U}$ has a $(4g-1)$-$\llfr$, then we have
$$\overline{Y^{-1}U}=\overline{\underbrace{\cdots b_{2g+1}\cdots b_{4g-1}}_{X_1^{-1}}b_{4g}\underbrace{(b_1\cdots b_{2g-1})^t}_{U}}.$$ It leads to a contradiction that $\overline{X}=\overline{b_{2g}(b_{2g+1}\cdots b_{4g-1})^{t_1}\underbrace{b_{2g-1}\cdots b_1\cdots}_{X_1}}$ is reducible.

If $\overline{UY}$ has a $(4g-1)$-$\llfr$, then we have $$\overline{UY}=\overline{\underbrace{(b_1\cdots b_{2g-1})^t}_{U}b_{2g}\underbrace{b_{2g+1}\cdots b_{4g-1}\cdots}_{X_1}}.$$ 
It follows that $\overline{X}=\overline{b_{2g}(b_{2g+1}\cdots b_{4g-1})^{t_1}\underbrace{b_{2g+1}\cdots b_{4g-1}\cdots }_{X_1}},$ which contradicts the maximality of $t_1$.

Therefore, conclusion (c) also holds. For the case that $\overline{X^{-1}U}$ contains a $(4g-1)$-$\llfr$, the same argument applies.
\end{proof}

\section{Proof of Theorem \ref{main thm1}}\label{sect 5}

Before proving Theorem \ref{main thm1}, we still need the following lemmas.

\begin{lem}\label{freely reduced presentions of u, v}
Let $G$ be a surface group with symmetric presentation (\ref{symmetric presentation}), and let $\overline{U}=\overline{u_1u_2\cdots u_m}, ~\overline{V}=\overline{v_1v_2\cdots v_n}\in \W(\gs)$ be any two words (not necessarily irreducible) and $u, v\in G$ be two elements. Then 
\begin{enumerate}
    \item If $U$ and $V$ are conjugate (as group elements) in $G$, then $m\equiv n \pmod{2}$;
    \item If $u, v$ are conjugate in $G$, then $|u|+|v|$ is always even. Moreover, we can represent
$$u=X^{-1}A_2A_1X, \quad v=Y^{-1}A_1A_2Y$$
(as group elements in $G$) for some irreducible words $\overline{X}, \overline{Y}\in \W(\gs)$, where both $\overline{A_1A_2}$ and $\overline{A_2A_1}$ are cyclic permutations of the normal form $\nf([u])$ of the conjugacy class $[u]$, and both $\overline{X^{-1}A_2A_1X}$ and $\overline{Y^{-1}A_1A_2Y}$ are freely reduced.
\end{enumerate}
\end{lem}

\begin{proof}
   (1) Since $U$ and $V$ are conjugate in $G$, there exists a word $\overline{C}\in \W(\gs)$ such that 
   $$U=C^{-1}VC$$ as group elements. Hence they share the same normal form
   $$\nf(U)=\nf(C^{-1}VC).$$
Note that the reductions ``$U\to \nf(U)$ and $C^{-1}VC\to \nf(C^{-1}VC)$'' are achieved by finitely many $S_{(i)}$'s, and each $S_{(i)}$ $(i=1,\ldots, 4)$ decreases the word length by an even number. Therefore, 
$$m\equiv|\nf(U)|=|\nf(C^{-1}VC)|\equiv|\overline{C^{-1}VC}|=2|\overline{C}|+|\overline{V}|\equiv n \pmod{2}.$$

(2) From item (1), we immediately obtain that $|u|+|v|$ is even. 
Furthermore, by \cite[Theorem 1.2]{WZZ25}, the conjugacy class $[u]=[v]$ has a unique normal form $\nf([u]):=\overline{A}$, which is a cyclically irreducible word in $[u]$. Moreover, every element $v\in [u]$ can be written
$$v=C_v^{-1}AC_v $$
(as group element) for some irreducible word $\overline{C_v}\in\W(\gs)$. Note that the word $\overline{C_v^{-1}AC_v}$ may be reducible. 
If $\overline{C_v^{-1}AC_v}$ is not freely reduced,  i.e., it contains subwords of type $S_{(1)}$, then after reducing it by finite times $S_{(1)}$, we can get an expression 
$$v=X_v^{-1}A_vX_v$$
such that the word $\overline{X_v^{-1}A_vX_v}$ is freely reduced. Note that $\overline{A_v}$ is a cyclic permutation of $\overline{A}$, and $\overline{X_v}$ is a subword of the irreducible word $\overline{C_v}$ and hence $\overline{X_v}$ is again irreducible. Similarly, we also have such an expression
$u=X_u^{-1}A_uX_u.$
Since both $\overline{A_u}$ and $\overline{A_v}$ are cyclic permutations of the normal form $\nf([u])=A$, we have $\overline{A_u}$ is a cyclic permutation of $\overline{A_v}$, and hence they can be decomposed as $\overline{A_u}=\overline{A_2A_1}$ and $\overline{A_v}=\overline{A_1A_2}$.
\end{proof}

\begin{lem}\label{lem lower bound}
    For any surface group $G$ with a symmetric presentation (\ref{symmetric presentation}) and any $n\in \N$, we have
    $$\cl(2n+1)=\cl(2n)\geq n-1.$$
\end{lem}

\begin{proof}
Obviously, $\cl(2n+1)=\cl(2n)$ follows form Lemma \ref{freely reduced presentions of u, v}. Let $u=c_1$ and $v=c_2^{1-n}c_1c_2^{n-1}$. Note that $|u|=1$ and $|v|=2n-1$. Then $|u|+|v|\leq 2n$ and $\cl(2n)\geq\cl(u,v)=n-1.$
\end{proof}

We now prove Theorem \ref{main thm1} through distance computations on Cayley graph.

\begin{thm}[Theorem \ref{main thm1}]\label{thm main result}
Let $G$ be a surface group with a symmetric presentation (\ref{symmetric presentation}).
Then the conjugator length function $\cl:\N\to\N$ of $G$ satisfies
$$n-1\leq \cl(2n)=\cl(2n+1)\leq n+8g-1$$
for any genus $g\geq 2$.
\end{thm}
\begin{proof}
The lower bound follows from Lemma \ref{lem lower bound}. We now prove the upper bound. 
 Let $u$ and $v$ be two conjugate elements in $G$ with length $|u|+|v|\leq 2n$.  By Lemma \ref{lem 4g-1-reduction 2}, Lemma \ref{lem for special case to cancel 4g-1} and Lemma \ref{freely reduced presentions of u, v}, we can obtain four irreducible words $\overline{X},\overline{Y},\overline{A_1}$ and $\overline{A_2}$ satisfying the following conditions:
 \begin{itemize}
     \item As group elements in $G$, $$u=X^{-1}A_2A_1X, \quad v=Y^{-1}A_1A_2Y.$$
     \item Both $\overline{A_1A_2}$ and $\overline{A_2A_1}$ are cyclic permutations of the normal form $\nf([u]):=\overline{A}$ of the conjugacy class $[u]=[v]$, and are therefore cyclically irreducible.
     \item Both $\overline{X^{-1}A_2A_1X}$ and $\overline{Y^{-1}A_1A_2Y}$ are freely reduced.
     \item None of $\overline{X^{-1}A_2A_1},~\overline{A_2A_1X},~\overline{Y^{-1}A_1A_2}$ and $\overline{A_1A_2Y}$ contains a $(4g-1)$-$\llfr$.
 \end{itemize}
 
Note that the reduction of $\overline{X^{-1}A_2A_1X}$  (resp. $\overline{Y^{-1}A_1A_2Y}$ ) to its normal form $\nf(u)$ (resp. $\nf(v)$) can be finished in the following two steps:
$$\overline{X^{-1}A_2A_1X}\xrightarrow{R_1} \overline{\nf(X^{-1}A_2A_1)X}\xrightarrow{R_2} \nf(u),$$
$$\overline{Y^{-1}A_1A_2Y}\xrightarrow{R'_1} \overline{\nf(Y^{-1}A_1A_2)Y}\xrightarrow{R'_2} \nf(v).$$
Each step reduces the length by
\begin{eqnarray*}
     R_1:=|X|+|A|-|X^{-1}A_2A_1|,&R_2:=|X^{-1}A_2A_1|+|X|-|u|,\\
     R_1':=|Y|+|A|-|Y^{-1}A_1A_2|,&R_2':=|Y^{-1}A_1A_2|+|Y|-|v|.
 \end{eqnarray*}
Summing the above four equations,  we have
\begin{eqnarray*}
   2(|X|+|Y|+|A|)-(R_1+R_2+R_1'+R_2')=|u|+|v|
    \leq 2n.
\end{eqnarray*}
Furthermore, since $\overline{Y^{-1}A_1X}$ is a conjugator of $u$ and $v$ (see Figure \ref{fig: conjugator}),  we have
\begin{eqnarray}\label{eq. upper bound for cl}
    \cl(u,v)&\leq&|Y^{-1}A_1X|\nonumber\\
    &\leq& |Y|+|A_1|+|X|\nonumber\\
    &\leq&|X|+|Y|+|A|-1\nonumber\\
    &\leq& n-1+\frac{1}{2}(R_1+R_2+R_1'+R_2').
\end{eqnarray}

\begin{figure}
     \centering
   \includegraphics[width=0.5\linewidth]{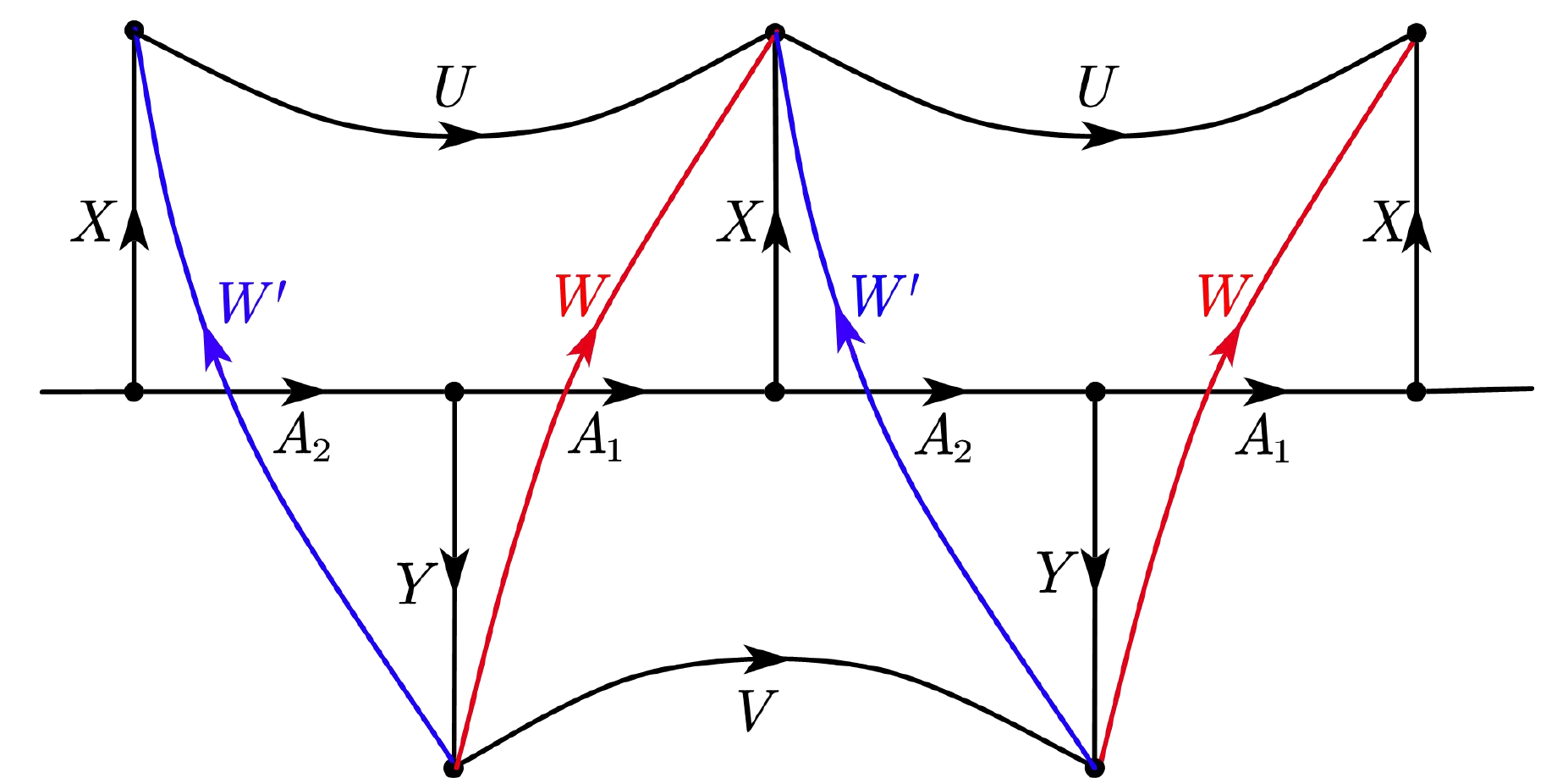}
     \caption{In this Cayley graph w.r.t. the symmetric presentation, as group elements, $W=Y^{-1}A_1X$, $W'=Y^{-1}A_2^{-1}X$ and $U=(W')^{-1}VW'=W^{-1}VW$. Therefore, $\overline{W}$ and $\overline{W'}$ represent two different conjugators of $U=u$ and $V=v$.}
     \label{fig: conjugator}
 \end{figure}

To complete the proof of Theorem \ref{main thm1}, it suffices to prove: 

\begin{claim}\label{claim 2.}
In Eq. (\ref{eq. upper bound for cl}), we have $R_1+R_2+R_1'+R_2'\leq 16g.$ 
\end{claim}

\begin{proof}[Proof of Claim 2.]
Applying Lemma \ref{lem observations only case of 4g-1} to the irreducible word $\overline{X}$ and the cyclically irreducible word $$\overline{A'}:=\overline{A_2A_1},$$
we have the following two cases.

   \textbf{Case (1)} Lemma \ref{lem observations only case of 4g-1}(a) holds. 
   
   Then we have the following reduction process for "$\overline{W_1}$" in Lemma \ref{lem observations only case of 4g-1}(a):
\begin{eqnarray*}
\overline{X^{-1}A'X}&=&\overline{\underbrace{X'^{-1} b_{4g}\cdots b_{2g+2}b_{2g+1}}_{X^{-1}}\underbrace{(b_{2g+2}\cdots b_{4g})^{t_0}b_1b_{2g+3}\cdots b_{4g}}_{A':=A_2A_1}\underbrace{b_1b_2\cdots b_{2g}X'}_{X}}\\
   &\xrightarrow{S_{(3,t_0)}}& \overline{\underbrace{X'^{-1} (b_{4g}\cdots b_{2g+2})^{t_0+1}b_{2g+3}\cdots b_{4g}}_{\nf(X^{-1}A')}\underbrace{b_1b_2\cdots b_{2g}X'}_{X}}\\
   &\xrightarrow{S_{(2,4g-1)}}&\overline{\underbrace{X'^{-1}(b_{4g}\cdots b_{2g+2})^{t_0}b_{4g}\cdots b_{2g+3}}_{B} \underbrace{b_1X'}_{C}}:=\overline{D}.
\end{eqnarray*}
Note that $\overline{B}\subset\nf(X^{-1}A')$,  hence it is irreducible. Furthermore, since $\overline{X}$ is irreducible, by Lemma \ref{claim 2},  $\overline{C}$ is also irreducible. Since there is no $\llfr$ at the junction of $\overline{BC}$ and $(\overline{B},\overline{C})$ does not satisfy the condition in \cite[Table 4]{WZZ25}, it follows that  $\overline{BC}=\overline{D}$ is irreducible. Therefore, we have 
\begin{equation}\label{eq. case 1 R}
R_1+R_2=2+(4g-2)=4g.    
\end{equation}
For the symmetric case "$\overline{W_2}$" in Lemma \ref{lem observations only case of 4g-1}(a), by similar arguments, Eq. (\ref{eq. case 1 R}) also holds.

\textbf{Case (2)} Lemma \ref{lem observations only case of 4g-1}(b) holds, i.e., neither $\overline{\nf(X^{-1}A')X}$ nor $\overline{X^{-1}\nf(A'X)}$ has a $(4g-1)$-$\llfr$. 

If $\overline{\nf(X^{-1}A')X}$ is freely reduced. Then by Lemma \ref{lem observation 2 of tables}, we have 
\begin{equation}\label{eq. 2.1R}
R_1+R_2\leq 4g+4g=8g.
\end{equation}

If $\overline{\nf(X^{-1}A')X}$ is not freely reduced and $\overline{A'}\neq\overline{(b_1\cdots b_{2g-1})^t}$,
then by Lemma \ref{lem freely reduced}(b), we have 
\begin{equation}\label{eq. 2.2R}
    R_1+R_2\leq 4g+2.
\end{equation}

If $\overline{\nf(X^{-1}A')X}$ is not freely reduced and $\overline{A'}=\overline{(b_1\cdots b_{2g-1})^t}$ for some $t\geq  1$ and $\overline{b_1\cdots b_{4g}}\in\RR$, then by Lemma \ref{lem used in final}, we obtain the irreducible word $\overline{X}=\overline{b_{2g}X_1}$ for some irreducible word $\overline{X_1}\neq \overline{b_{4g}\cdots}$. Therefore,
\begin{eqnarray}\label{eq ii-2}  
\overline{X^{-1}A'X}&=&\overline{\underbrace{X^{-1}_1b_{4g}}_{X^{-1}}\underbrace{(b_1\cdots b_{2g-1})^{t}}_{A'}\underbrace{b_{2g}X_1}_{X}}\notag\\
&\xrightarrow{S_{(3)}}&\overline{X_1^{-1}\underbrace{(b_{2g-1}\cdots b_1)^t}_{B}X_1}.
\end{eqnarray}
Furthermore, if $\overline{X_1^{-1}BX_1}$ is freely reduced, then by Lemma \ref{lem used in final}, we have $\overline{X_1^{-1}BX_1}$ is irreducible and thus 
\begin{equation}\label{R_1+R_2=2}
   R_1+R_2=2. 
\end{equation} 
If $\overline{X_1^{-1}BX_1}$ is not freely reduced, then we have two subcases.
\begin{enumerate} 
    \item[(a)] $\overline{X_1}=\overline{(b_{2g+1}\cdots b_{4g-1})^{t_0}b_{2g+1}\cdots b_{2g+i}X_2}$ with $t_0\geq 0$ maximal firstly and $1\leq i\leq 2g-2$ maximal secondly, for some irreducible word  $\overline{X_2}\neq \overline{b_{2g+i+1}\cdots }$ or $\overline{b_i\cdots}$.
Since $\overline{X^{-1}A'X}$ has no $(4g-1)$-$\llfr$, we have $t_0=0$. Then we have a reduction
\begin{eqnarray}\label{eq, iS(1)}
    \overline{X_1^{-1}BX_1}&=&\overline{\underbrace{X_2^{-1}b_{i}\cdots b_1}_{X_1^{-1}}\underbrace{(b_{2g-1}\cdots b_1)^t}_{B}\underbrace{b_{2g+1}\cdots b_{2g+i}X_2}_{X_1}}\notag\\
    &\xrightarrow{i\cdot S_{(1)}}&\overline{\underbrace{X_2^{-1}b_{i}\cdots b_1}_{X_1^{-1}}\underbrace{(b_{2g-1}\cdots b_1)^{t-1}b_{2g-1}\cdots b_{i+1}}_{ C}X_2}:=\overline{D}.
\end{eqnarray}
Note that $\overline{D}$ is freely reduced because $\overline{X_2}\neq b_{2g+i+1}$. Furthermore, since there is no $\llfr$ at the junction of $\overline{X_1^{-1}C}$ and $(\overline{X_1^{-1}},\overline{C})$ has no reducing-subword pair satisfying \cite[Table 4]{WZZ25}, we have $\overline{X_1^{-1}C}:=\overline{E}$ is irreducible. Moreover, $\overline{EX_2}$ has no $(4g-1)$-$\llfr$ because $\overline{X_2}\neq \overline{b_i\cdots}$. Therefore, by Lemma \ref{lem observation 2 of tables}, we have 
\begin{equation}\label{eq. EX2}
  |E|+|X_2|-|EX_2|\leq 4g.
\end{equation}
Combining Eqs. (\ref{eq ii-2})(\ref{eq, iS(1)}) and (\ref{eq. EX2}), we obtain
\begin{equation}\label{eq. R1+R2. 4}
    R_1+R_2\leq  2+2\cdot i+4g\leq  2+2(2g-2)+4g=8g-2.
\end{equation}

\item[(b)] $\overline{X_1}=\overline{(b_{2g-1}\cdots b_{1})^{t_0}b_{2g-1}\cdots b_{i}X_2}$ with $t_0$ maximal firstly and $2\leq i\leq 2g-1$ minimal secondly, for some $\overline{X_2}$. The argument is similar to the above Case (a).
\end{enumerate}

In conclusion, for Case (2), by Eqs. (\ref{eq. 2.1R})(\ref{eq. 2.2R})(\ref{R_1+R_2=2}) and (\ref{eq. R1+R2. 4}), we have 
\begin{equation}\label{eq. case 2 R}
    R_1+R_2\leq \max\{8g, 4g+2, 2, 8g-2\}=8g.
\end{equation} 
\renewcommand{\qedsymbol}{}
\end{proof}
Finally, combining Eqs. (\ref{eq. case 1 R}) and (\ref{eq. case 2 R}), we have $R_1+R_2\leq 8g.$ Similarly, we also have $R_1'+R_2'\leq 8g$. 
Therefore, 
$$R_1+R_2+R_1'+R_2'\leq 16g.$$ 
The proof is complete.
\end{proof}

\begin{rem}
For free groups, Lemma \ref{freely reduced presentions of u, v}  and Lemma \ref{lem lower bound} also hold clearly. Moreover, since $S_{(1)}$ is the unique reduction operator in free groups,  the words $\overline{X^{-1}A_2A_1X}$ and $\overline{Y^{-1}A_1A_2Y}$ in Lemma \ref{freely reduced presentions of u, v}(2) are indeed irreducible. Therefore, $R_i=R'_i=0 (i=1,2)$ in the proof of Theorem \ref{thm main result} and hence we can obtain 
\begin{equation}\label{CL=n-1}
   \cl(2n+1)=\cl(2n)=n-1 
\end{equation} 
for any $n\in\N$ in a free group.
\end{rem}
\begin{ques}
Does the above Eq. (\ref{CL=n-1}) hold in a surface group?
\end{ques}



\begin{thebibliography}{dd}
\bibitem{AB23} C.~R. Abbott and J.~A. Behrstock, \emph{Conjugator lengths in hierarchically hyperbolic groups}, Groups Geom. Dyn. {\bf 17} (2023), no.~3, 805--838.

\bibitem{AS16} Y. Antol\'in and A.~W. Sale, \emph{Permute and conjugate: the conjugacy problem in relatively hyperbolic groups}, Bull. Lond. Math. Soc. {\bf 48} (2016), no.~4, 657--675.

\bibitem{BM23} J.~M. Belk and F. Matucci, \emph{Conjugator length in Thompson's groups}, Bull. Lond. Math. Soc. {\bf 55} (2023), no.~2, 793--810.

\bibitem{BH05} M.~R. Bridson and J. Howie, \emph{Conjugacy of finite subsets in hyperbolic groups}, Internat. J. Algebra Comput. {\bf 15} (2005), no.~4, 725--756.

\bibitem{BR25a} M. Bridson and T. Riley, \emph{The lengths of conjugators in the model filiform groups}, (2025), preprint, \href{https://arxiv.org/abs/2506.01235}{arXiv:2506.01235}.

\bibitem{BR25b} M. Bridson and T. Riley, \emph{Linear Diophantine equations and conjugator length in 2-step nilpotent groups}, (2025), preprint, \href{https://arxiv.org/abs/2506.01239}{arXiv:2506.01239}.

\bibitem{BR25c} M. Bridson and T. Riley, \emph{Snowflake groups and conjugacy length functions with non-integer exponents}, (2025), preprint, \href{https://arxiv.org/abs/2512.14038}{arXiv:2512.14038}.

\bibitem{BR25d} M. Bridson and T. Riley, \emph{Groups with fast-growing conjugator length functions}, (2025), preprint, \href{https://arxiv.org/abs/2512.23674v1}{arXiv:2512.23674}.

\bibitem{De12} M. Dehn, \emph{Transformation der Kurven auf zweiseitigen Fl\"achen}, Math. Ann. {\bf 72} (1912), no.~3, 413--421. 


\bibitem{Ly89} I.~G. Lys\"enok, \emph{Some algorithmic properties of hyperbolic groups}, Math. USSR-Izv. {\bf 35} (1990), no.~1, 145--163; translated from Izv. Akad. Nauk SSSR Ser. Mat. {\bf 53} (1989), no.~4, 814--832, 912.

\bibitem{Sa12} A.~W. Sale, \emph{The length of conjugators in solvable groups and lattices of semisimple Lie groups}, ProQuest LLC, Ann Arbor, MI, 2012.

\bibitem{WZZ25} K. Wang, Q. Zhang and X. Zhao, \emph{Word length formulae and normal forms of conjugacy classes in surface groups}, (2025), 61pp, preprint, \href{https://arxiv.org/abs/2511.12862v2}{arXiv:2511.12862v2}.

\end{thebibliography}
\end{document}